\documentclass[11pt]{amsart}

\usepackage{latexsym,rawfonts}
\usepackage{amsfonts,amssymb}
\usepackage{amsmath,amsthm}

\usepackage[plainpages=false]{hyperref}
\usepackage{graphicx}
\usepackage{color}

\newtheorem{theorem}{Theorem}

\newtheorem{lemma}{Lemma}
\newtheorem{proposition}{Proposition}

\theoremstyle{definition}
\newtheorem{definition}{Definition}

\newcommand{\beq}{\begin{equation}}
\newcommand{\eeq}{\end{equation}}
\newcommand{\beqs}{\begin{eqnarray*}}
\newcommand{\eeqs}{\end{eqnarray*}}
\newcommand{\beqn}{\begin{eqnarray}}
\newcommand{\eeqn}{\end{eqnarray}}
\newcommand{\beqa}{\begin{array}}
\newcommand{\eeqa}{\end{array}}

\usepackage{amsmath}
\numberwithin{equation}{section}
\numberwithin{theorem}{section}
\numberwithin{lemma}{section}
\numberwithin{remark}{section}
\numberwithin{assumption}{section}
\numberwithin{definition}{section}
\numberwithin{fact}{section}
\numberwithin{question}{section}
\numberwithin{proposition}{section}

\begin{document}

\title{The Dirichlet problem for mixed Hessian equations on Hermitian manifolds}

\author{Qiang Tu}
\address{Faculty of Mathematics and Statistics, Hubei Key Laboratory of Applied Mathematics, Hubei University,  Wuhan 430062, P.R. China}
\email{qiangtu@hubu.edu.cn}

\author{Ni Xiang$^{\ast}$}
\address{Faculty of Mathematics and Statistics, Hubei Key Laboratory of Applied Mathematics, Hubei University,  Wuhan 430062, P.R. China}
\email{nixiang@hubu.edu.cn}


\date{}



\thanks{This research was supported by funds from  Natural Science Foundation of Hubei Province, China, No. 2020CFB246 and the National Natural Science Foundation of China No. 11971157, 12101206.}
\thanks{$\ast$ Corresponding author}

\begin{abstract}
In this paper we study the Dirichlet problem for a class of Hessian type equation with its structure as a combination of elementary symmetric functions on Hermitian manifolds. Under some conditions with the initial data on manifolds and admissible subsolutions, we derive a priori estimates for this complex mixed Hessian equation and solvability of the corresponding Dirichlet problem.

\end{abstract}

\maketitle
\vskip4ex

\section{Introduction}

Let $(M, \omega)$ be a closed Hermitian manifold of complex dimension $n\geq2$ with smooth boundary $\partial M$ and $\bar{M}=M\cup \partial M$,
 fix a real smooth closed $(1,1)$-form $\chi_0$ on $M$. For any $C^2$
function $u: M \rightarrow \mathbb{R}$, we can obtain a new real $(1,1)$-form
\begin{eqnarray*}
\chi_u=\chi_0+\frac{\sqrt{-1}}{2}\partial\overline{\partial} u.
\end{eqnarray*}
We consider the following Dirichlet problem of Hessian type equation on $(M, \omega)$
\begin{equation}\label{K-eq}
\left\{
\begin{aligned}
&\chi^{k}_{u}\wedge \omega^{n-k}=\sum_{l=0}^{k-1}\alpha_l(z)\chi^{l}_{u}\wedge \omega^{n-l}, && in~ M,\\
&u=\varphi && on~ \partial M,
\end{aligned}
\right.
\end{equation}
where $1<k\leq n$, $\alpha_l(z)$ and $\varphi$ are real smooth functions on $M$ and $\partial M$, respectively.
Note that this is a class of fully nonlinear equation with its structure as a combination of elementary symmetric functions on  Hermitian manifolds.
In order to keep the ellipticity of  equation \eqref{K-eq}, we require  the eigenvalues of $(1,1)$-form  $\chi_u$ with respect to $\omega$ belong to
the G{\aa}rding's cone $\Gamma_{k-1}$.  Hence we introduce the following definition.
\begin{definition}\label{def-1}
A function $u\in C^2(M)$ is called $k$-admissible  if $\chi_u \in \Gamma_k(M)$ for any $z\in M$,
 where $\Gamma_k(M)$ is the G{\aa}rding cone
\begin{eqnarray*}\label{cone}
\Gamma_k(M)=\{ \chi  \in \mathcal{A}^{1,1}(M):\sigma _i (\lambda[\chi])>0, \ \forall \ 1 \le i \le k\}.
\end{eqnarray*}
\end{definition}
{}
The equation in \eqref{K-eq}
\begin{eqnarray}\label{K-eq-11}
\chi^{k}_{u}\wedge \omega^{n-k}=\sum_{l=0}^{k-1}\alpha_l(z)\chi^{l}_{u}\wedge \omega^{n-l}
\end{eqnarray}
includes some of the most partial differential equations in complex geometry and analysis. When $k=n$ and $\alpha_1=\cdots=\alpha_{n-1}=0$, the equation becomes the complex Monge-Amp\`ere equation $\chi^{n}_{u}=\alpha_0 \omega^n$, which was solved by Yau \cite{Yau78} on closed K\"ahler manifolds in the resolution of the Calabi conjecture. Then
Tosatti-Weinkove \cite{Tos10,Tos19} have solved the analogous problem for the  equation on closed Hermitian manifolds.  The corresponding Dirichlet problems on  manifolds were studied by  Cherrier-Hanani \cite{Che87} and Guan-Li \cite{Guan10, Guan12, Guan13}.

In fact, the Hessian equation $\chi^{k}_{u}\wedge \omega^{n-k}=\alpha_0 \omega^n$
and the Hessian quotient equation $\chi^{k}_{u}\wedge \omega^{n-k}=\alpha_l(z)\chi^{l}_{u}\wedge \omega^{n-l}$ are also the special case of \eqref{K-eq-11}.
For equation $\chi^{k}_{u}\wedge \omega^{n-k}=\alpha_0 \omega^n$, Hou-Ma-Wu \cite{Hou10} established  the second order estimates for the equation without boundary on K\"ahler manifold, and then Dinew-Kolodziej \cite{Din12} solved the equation by combining the Liouville theorem and Hou-Ma-Wu's results. Zhang\cite{Zhang17} and Sz\'ekelyhidi  \cite{Sze18} have solved the equation without boundary  on Hermitian manifolds. The corresponding Dirichlet problem on  manifolds has also attracted the interest of many researchers, such as Gu-Nguyen \cite{GN-18} were able to obtain continuous solutions to the equation with boundary on Hermitian manifolds, and Collins-Picard \cite{Co19} solved the  problem  under the existence of a subsolution.  For $(k,l)$-Hessian quotient equation $\chi^{k}_{u}\wedge \omega^{n-k}=\alpha_l(z)\chi^{l}_{u}\wedge \omega^{n-l}$, the $(n,n-1)$-Hessian quotient equation have appeared in a problem proposed by Donaldson in the setting of moment maps and  was solved by Song-Weinkove \cite{Song08}.
Then $(n, l)$-Hessian quotient equation  was considered by Fang-Lai-Ma \cite{Fang11} on K\"ahler manifold, and by Guan-Li \cite{Guan12}, Guan-Sun \cite{Guan15}  on Hermitian manifolds.
The general $(k,l)$-Hessian quotient equation with $k<n$ without boundary on Hermitian manifolds was studied by  Sz\'ekelyhidi \cite{Sze18} for constants $\alpha_l$ and by  Sun \cite{Sun17} for functions $\alpha_l$.  The corresponding Dirichlet problem on Hermitian manifolds was studied by Feng-Ge-Zheng \cite{Feng20}, since they can only obtain the gradient estimates in some special cases, the existence of solution can be solved in these kinds of special cases.

When $k=n$ and $\alpha_l \in \mathbb{R}$, equation \eqref{K-eq-11} was raised as a conjecture by Chen \cite{Chen00} in the study of Mabuchi energy. The conjecture was solved by Collins-Sz\'ekelyhidi \cite{Co17} for some  special constant  $\alpha_l$.  Later, Phong-T\^o \cite{Ph17-1} generalized  Collins-Sz\'ekelyhidi's result for nonnegative constants $\alpha_l$. Moreover, a generalized equation of Chen's problem was studied by Sun \cite{Sun14-1,Sun14-2} and Pingali \cite{ Pin14-1, Pin14-2,Pin16}.

The initial motivation of our work is the following: As an important example for the applications of the general notion of fully nonlinear elliptic equations, Krylov \cite{Kry95} studied the Dirichlet problem of the equation
\begin{eqnarray}\label{K-eq-122}
\sigma_k(D^2u)=\sum_{l=0}^{k-1}\alpha_l(x) \sigma_l(D^2 u)
\end{eqnarray}
in a $(k-1)$-convex domain in $\mathbb{R}^n$ with $\alpha_l>0 (0\leq l\leq k-1)$. Recently, Guan-Zhang \cite{GZ19} observed that equation \eqref{K-eq-122} can be rewritten as   the following equation
\begin{eqnarray}\label{K-eq-13}
\frac{\sigma_k}{\sigma_{k-1}}(D^2u) -\sum_{l=0}^{k-2}\alpha_l(x) \frac{\sigma_l}{\sigma_{k-1}}(D^2 u) =-\alpha_{k-1}.
\end{eqnarray}
Actually,  the equation is elliptic and concave in $\Gamma_{k-1}$. Then they obtained a priori $C^2$ estimate of the $(k-1)$-admissible solution of equation \eqref{K-eq-13} without sign requirement for $\alpha_{k-1}$ and solved  the Dirichlet
problem for the corresponding equation.  Later the corresponding Neumann problem and prescribed  curvature equations were also discussed in \cite{CCX19, CCX21, Zhou22,Chen20}. 
Recently, Zhang \cite{Zhang21}  cosidered the Dirichlet problem for \eqref{K-eq-122} on complex domains in $\mathbb{C}^n$, which can be seen as the extension of Guan-Zhang's result. Chen \cite{Chen21} and Zhou \cite{Zhou21} provide a sufficient and necessary condition for the solvability of the mixed hessian equation on  K\"ahler manifold without boundary. A natural problem is raised whether we can consider the Dirichlet problem for  equation \eqref{K-eq-122} on  complex  manifolds. The following is our main result.

\begin{theorem}\label{Main}
Let $(M, \omega)$ be a closed Hermitian manifold of complex dimension $n\geq2$,
$\chi_0 \in \Gamma_{k-1}(M)$ be a $(1,1)$-form, $\alpha_{k-1}(z), \varphi(z), \alpha_{l}$ be  smooth functions and $\alpha_{l}(z)> 0$ for $l=0,1,\cdots,k-2$.
 Suppose  there exists an $(k-1)$-admissible subsolution $\underline{u}\in C^2(\bar{M})$ such that
\begin{equation}\label{sub-condition}
\left\{
\begin{aligned}
&\chi^{k}_{\underline{u}}\wedge \omega^{n-k}\geq\sum_{l=0}^{k-1}\alpha_l(z)\chi^{l}_{\underline{u}}\wedge \omega^{n-l}, && in~ M,\\
&\underline{u}=\varphi && on~ \partial M.
\end{aligned}
\right.
\end{equation}
Then  there exists a unique $(k-1)$-admissible solution $u\in C^{\infty}(\bar{M})$ of Dirichlet problem \eqref{K-eq}. Moreover, we have
\begin{eqnarray*}
\|u\|_{C^2(\bar{M})} \leq C,
\end{eqnarray*}
where the constant $C$ depends on $M, \omega, \chi_0, \varphi$, $\|\underline{u}\|_{C^2}$, $\inf_{\bar{M}} \alpha_l$ with $0\leq l \leq k-2$ and $\|\alpha_l\|_{C^2}$ with $0\leq l \leq k-1$.
\end{theorem}

Note that, equation \eqref{K-eq} can be seen a class of fully nonlinear elliptic  equations  and there is no sign requirement of the coefficient function $\alpha_{k-1}$.
Using some notaions on Hermitian manifolds, equation  \eqref{K-eq} can be rewritten  as a combination of elementary symmetric functions on Hermitian manifolds,  which is similar to equation \eqref{K-eq-13}.  Therefore, we will construct barrier functions  to obtain a prior estimates for equation \eqref{K-eq}  and solve  the Dirichlet problem.  However, compared with the classical method, we apply a blow-up argument and Liouville-type theorem due to Dinew-Kolodziej \cite{Din12} to obtain the gradient estimate. For this purpose, we follow a technique of Hou-Ma-Wu \cite{Hou10} (also see \cite{Co19}) to derive a second derivative bound of the form
$$\sup_{\bar{M}}|\partial \bar{\partial} u|\leq C(1+ \sup_{M} |\nabla u|^2).$$

The organization of the paper is as follows. In Section 2 we start with some preliminaries. In Section 3 we prove  $C^0$ estimates and 
second order derivative interior estimates. Boundary second order derivative estimates are given in Seciton 4 and 5. In Section 6, we use a blow-up argument and Liouville-type theorem  to obtain gradient estimates and  Theorem 1.1 is proved by the  standard continuity method.

\section{Preliminaries}

In this section, we give some basic properties of elementary symmetric functions, which could be found in
\cite{L96, S05}, and establish some key lemmas. Throughout this paper, repeated indices will be summed unless otherwise stated.

\subsection{Basic properties of elementary symmetric functions}

For $\lambda=(\lambda_1, ... , \lambda_n)\in \mathbb{R}^n$,
the $k$-th elementary symmetric function is defined
by
\begin{equation*}
\sigma_k(\lambda) = \sum _{1 \le i_1 < i_2 <\cdots<i_k\leq n}\lambda_{i_1}\lambda_{i_2}\cdots\lambda_{i_k}.
\end{equation*}
We also set $\sigma_0=1$ and denote by $\sigma_k(\lambda \left| i \right.)$ the $k$-th symmetric
function with $\lambda_i = 0$.
The generalized Newton-MacLaurin inequality and some well-known result (See \cite{And94}) are as follows, which will be used later.

\begin{proposition}\label{prop2.4}
For $\lambda \in \Gamma_k:=\{ \lambda  \in \mathbb{R}^n :\sigma _i (\lambda ) > 0, \ \forall \ 1 \le i \le k\}$ and $n\geq k>l\geq 0$, $ r>s\geq 0$, $k\geq r$, $l\geq s$, we have
\begin{align}\label{NM}
\Bigg[\frac{{\sigma _k(\lambda )}/{C_n^k }}{{\sigma _l (\lambda )}/{C_n^l }}\Bigg]^{\frac{1}{k-l}}
\le \Bigg[\frac{{\sigma _r (\lambda )}/{C_n^r }}{{\sigma _s (\lambda )}/{C_n^s }}\Bigg]^{\frac{1}{r-s}}.
\end{align}
\end{proposition}

\begin{proposition}
Let $A=(a_{ij})$ be a Hermitian matrix, $\lambda(A)=(\lambda_1, \cdots, \lambda_n)$ be the eigenvalues of $A$ and
$F=F(A)=f(\lambda(A))$ be a symmetric function of $\lambda(A)$. Then for any Hermitian matrix $B= (b_{ij})$, we
have
\begin{eqnarray}\label{f-2}
\frac{\partial^2 F}{\partial a_{ij}\partial
a_{st}}b_{ij}b_{st}=\frac{\partial^2 f}{\partial\lambda_p
\partial\lambda_q}b_{pp}b_{qq}+2\sum_{p<q}\frac{\frac{\partial
f}{\partial \lambda_p}-\frac{\partial f}{\partial
\lambda_q}}{\lambda_p-\lambda_q}b^{2}_{pq}.
\end{eqnarray}
In addition, if $f$ is concave and $\lambda_1\leq \lambda_2\leq ... \leq\lambda_n$, then we have
\begin{eqnarray}\label{f-3}
f_1\geq f_2\geq ... \geq f_n,
\end{eqnarray}
where $f_i=\frac{\partial f}{\partial \lambda_i}.$
\end{proposition}

\begin{proposition}\label{mix-pro-54}
Let $W={W_{ij}}$ be an $n\times n$ symmetric matric and $\lambda(W)=(\lambda_{1},\lambda_{2},\cdots,\lambda_{n})$ be the eigenvalues of the symmetric matric $W$. Suppose that $W={W_{ij}}$ is diagonal and $\lambda_{i}=W_{ii}$,then we have
$$\frac{\partial{\lambda_{i}}}{\partial{W_{ii}}}=1,\quad \frac{\partial{\lambda_{k}}}{\partial{W_{ij}}}=0\quad otherwise,\leqno(2.1.5)$$
$$\frac{\partial^{2} \lambda_{i}}{\partial{W_{ij}}\partial{W_{ji}}}=\frac{1}{\lambda_{i}-\lambda_{j}}\quad \mbox{for}~ i\neq j \quad and \quad \lambda_{i}\neq \lambda_{j}.\leqno(2.1.6)$$
$$\frac{\partial^{2} \lambda_{i}}{\partial{W_{kl}}\partial{W_{pq}}}=0 \quad otherwise.\leqno(2.1.7)$$
\end{proposition}

\begin{lemma}
Let $\lambda(z) \in C^0(M, \mathbf{R}^n) \cap \Gamma_{k-1}$ satisfy
\begin{equation}\label{eq-mix-03}
f(\lambda(z),z):=\frac{\sigma_k(\lambda(z))}{\sigma_{k-1}(\lambda(z))}- \sum_{l=0}^{k-1} \beta_l(z) \frac{\sigma_l(\lambda(z))}{\sigma_{k-1}(\lambda(z))} =\beta(z),
\end{equation}
 where $\beta_l, \beta$ are smooth function with $\beta_l\geq 0$, then
\begin{equation}\label{ineq-mix-02}
\sum_{i=1}^n f_i(\lambda) \mu_i \geq f(\mu) +(k-l) \sum_{l=0}^{k-2} \beta_l \frac{\sigma_l(\lambda)}{\sigma_{k-1}(\lambda)}.
\end{equation}
\end{lemma}
\begin{proof}
Since $f$ is concave in $\Gamma_{k-1}$, we have
\begin{eqnarray*}
f(\mu)\leq \sum_{i=1}^nf_i(\lambda)(\mu_i-\lambda_i)+f(\lambda),
\end{eqnarray*}
which implies
\begin{eqnarray*}
\sum_{i=1}^nf_i(\lambda)\mu_i\geq f(\mu)+ \sum_{l=0}^{k-2} (k-l) \beta_l \frac{\sigma_l(\lambda)}{\sigma_{k-1}(\lambda)}.
\end{eqnarray*}
So, the proof is completed.
\end{proof}

\begin{lemma}\label{lemm-mix-23}
Let $\lambda(z) \in C^0(M, \mathbb{R}^n) \cap \Gamma_{k-1}$ satisfy \eqref{eq-mix-03} and $\beta_l, \beta$ be smooth function with $\beta_l\geq 0$. Assume $\mu \in C^0(M, \mathbb{R}^n) \cap \Gamma_{k-1}$ satisfy
\begin{equation}\label{lem-mix-71}
\frac{\sigma_{k-1}(\mu|i)}{\sigma_{k-2}(\mu|i)}- \sum_{l=1}^{k-2} \beta_l(z)\frac{\sigma_{l-1}(\mu|i)}{\sigma_{k-2}(\mu|i)}>\beta(z)    \quad \forall z\in M.
\end{equation}
Then there exist constants $N, \theta>0$ depending on $\|\mu\|_{C^0(M)}$, $\|\beta\|_{C^0(M)}$ and  $\|\beta_l\|_{C^0(M)}$ such that if
\begin{equation*}
 \lambda_{\max}(z):=\max_{1\leq i \leq n} \{\lambda_i(z)\} \geq N
\end{equation*}
we have at $z$
\begin{equation}\label{lem-mix-72}
\sum_i f_i(\lambda) (\mu_i -\lambda_i) \geq \theta+\theta \sum_i f_i(\lambda)
\end{equation}
or
\begin{equation}\label{lem-mix-73}
f_{max} \lambda_{\max} \geq \theta,
\end{equation}
where $f_{\max}=\frac{\partial f}{\partial \lambda_{\max}}$.
\end{lemma}
\begin{proof}
The proof can be seen in \cite[Lemma 2.7]{Chen21}.
\end{proof}

Let $\mathcal{A}^{1,1}(M)$ be
the space of real smooth $(1, 1)$-forms on the Hermitian manifold $(M, \omega)$.
For any $\chi \in \mathcal{A}^{1,1}(M)$, we write in a local coordinate chart $(z^1, ..., z^n)$
\begin{eqnarray*}
\omega=\frac{\sqrt{-1}}{2}g_{i\overline{j}}dz^i\wedge d\overline{z}^j
\end{eqnarray*}
and
\begin{equation*}
\chi=\frac{\sqrt{-1}}{2}\chi_{i\overline{j}}dz^i\wedge d\overline{z}^{j}.
\end{equation*}
In particular, in a local normal coordinate
system $g_{i\overline{j}}=\delta_{i\overline{j}}$,
the matrix $(\chi_{i\overline{j}})$ is a Hermitian matrix. We denote $\lambda(\chi_{i\overline{j}})$ by the
eigenvalues of the matrix $(\chi_{i\overline{j}})$.
We define $\sigma_k(\chi)$ with
respect to $\omega$ as
\begin{equation*}
\sigma_k(\chi)=\sigma_k(\lambda(\chi_{i\overline{j}})),
\end{equation*}
and the G{\aa}rding's cone on $M$ is defined by
\begin{equation*}
\Gamma_k(M)=\{ \chi  \in \mathcal{A}^{1,1}(M):\sigma _i (\chi)>0, \ \forall \ 1 \le i \le k\}.
\end{equation*}
In fact, the definition of $\sigma_k(\chi)$ is independent of the choice of local normal coordinate system, and it  can be
defined without the use of local normal coordinate by
\begin{equation*}
\sigma_{k}(\chi)=C_{n}^{k}\frac{\chi^k\wedge \omega^{n-k}}{\omega^n},
\end{equation*}
where $C_{n}^{k}=\frac{n!}{(n-k)!k!}$.

Using the above notation, we can rewrite equation \eqref{K-eq} as the following local form:
\begin{eqnarray}\label{K-eq1}
\left\{
\begin{aligned}
& \frac{ \sigma_k(\chi_{u})}{\sigma_{k-1}(\chi_{u})}-\sum_{l=0}^{k-2} \beta_l(z) \frac{\sigma_l(\chi_{u})}{\sigma_{k-1}(\chi_{u})}=\beta(z), && in~ M,\\
&u=\varphi && on~ \partial M,
\end{aligned}
\right.
\end{eqnarray}
where $\beta_l(z)=\frac{C_{n}^{k}}{C_n^l}\alpha_l(z)$ for $0\leq l \leq k-2$ and $\beta(z)=\frac{C_{n}^{k}}{C_n^{k-1}}\alpha_{k-1}(z)$.

According to an important observation by Guan-Zhang in \cite{GZ19}, we know that

\begin{proposition}
If $u\in C^2(M)$ with $\chi_u \in \Gamma_{k-1}(M)$ and $\beta_l(z)\geq 0$ for $0\leq l \leq k-2$, then the operator
\begin{equation}
G(\chi_u)=\frac{ \sigma_k(\chi_{u})}{\sigma_{k-1}(\chi_{u})}-\sum_{l=0}^{k-2} \beta_l(z) \frac{\sigma_l(\chi_{u})}{\sigma_{k-1}(\chi_{u})}
\end{equation}
is elliptic and concave.
\end{proposition}

For the convenience of notations, we will denote
$$G(\chi):=\frac{ \sigma_k(\chi)}{\sigma_{k-1}(\chi)}-\sum_{l=0}^{k-2} \beta_l(z) \frac{\sigma_l(\chi)}{\sigma_{k-1}(\chi)},$$
$$G_k(\chi):=\frac{ \sigma_k(\chi)}{\sigma_{k-1}(\chi)}, \quad G_l(\chi):=-\frac{ \sigma_l(\chi)}{\sigma_{k-1}(\chi)},$$
and
$$G^{i \bar{j}}:=\frac{\partial G}{\partial \chi_{i\bar{j}}}, G^{i \bar{j}, r \bar{s}}:=\frac{\partial^2 G}{\partial \chi_{i\bar{j}}\partial \chi_{r\bar{s}}}, G^{i \bar{j}}_l:=\frac{\partial G_l}{\partial \chi_{i\bar{j}}}, G^{i \bar{j}, r \bar{s}}_l:=\frac{\partial^2 G_l}{\partial \chi_{i\bar{j}}\partial \chi_{r\bar{s}}}$$
for any $1\leq i, j, r, s\leq n$ and $0\leq l \leq k-2$.

\begin{lemma}
If $u\in C^2(M)$ with $\chi_u \in \Gamma_{k-1}(M)$  satisfy
$$G(\chi_u)=\beta, \quad \mbox{in} ~M,$$
where  $\beta_l>0$ for $0\leq l\leq k-2$. Then
\begin{equation}\label{hq-ineq-1}
0<\frac{\sigma_l}{\sigma_{k-1}} \left(\lambda(\chi_u)\right)\leq C(n, k, inf_M \beta_l, \sup_M|\beta|), \quad 0\leq l \leq k-2;
\end{equation}
\begin{equation}\label{hq-ineq-2}
-\sup_M |\beta|<\frac{\sigma_k}{\sigma_{k-1}} \left(\lambda(\chi_u)\right)\leq C(n, k, \sum_{l=0}^{k-1} \sup_M |\beta_l|);
\end{equation}
\begin{equation}\label{hq-ineq-3}
\frac{n-k+1}{k}\leq \sum_i G^{i\bar{i}} \leq n-k-1 +\frac{(n-k+2)\sigma_{k-2}}{\sigma_{k-1}}  \left(\lambda(\chi_u)\right) \beta;
\end{equation}
\begin{equation}\label{hq-ineq-4}
 \sum_i G^{i\bar{i}}  \lambda_i (\chi_u) =\beta+\sum_{l=0}^{k-2} (k-l) \beta_l \frac{\sigma_l}{\sigma_{k-1}}\left(\lambda(\chi_u)\right).
\end{equation}
\end{lemma}
\begin{proof}
The proofs of Lemma are quite similar to that given for \cite[Lemma 2.8, 2.9]{Chen21} and so they are omitted.
\end{proof}

\section{$C^0$ estimates and Second order interior estimates}

\subsection{$C^0$ estimates}
\begin{theorem}\label{C0}
Let $u\in C^{\infty}(\bar{M})$ be an $(k-1)$-admissible solution for equation \eqref{K-eq}.
Under the assumptions mentioned  in Theorem \ref{Main},  then there exists a positive constant $C$ depending only on  $(M, \omega), \chi_0$, $\alpha_l, \varphi$ and the subsolution $\underline{u}$ such that
 $$\sup_{ \bar{M}} |u|\leq C.$$
 \end{theorem}

\begin{proof}
On the one hand, we know that the subsolution $\underline{u}\in C^2(\bar{M})$ satisfies $G(\chi_{\underline{u}})-G(\chi_u)\geq 0$ in $M$ and $\underline{u}-u=0$ on $\partial M$. By the ellipticity of $G$ and the maximum principle,
$$\underline{u}(z) \leq u(z), \quad \forall z\in \bar{M}. $$
On the other hand, let $v$ be a function satisfying
$$ \omega^{i\bar{k}} \left((\chi_0)_{\bar{k}i}+\partial_i \partial_{\bar{k}} v\right) =0 \quad \mbox{in}~M; \quad v=\varphi \quad \mbox{on}~ \partial M.$$
 Note that $\chi_u \in \Gamma_k(M) \subset \Gamma_1(M) $. By the comparison principle,
 $$ u(z) \leq v(z), \quad \forall z\in \bar{M}. $$
\end{proof}

According to the proof of Theorem \ref{C0}, we know that $u-v$ and $\underline{u}-u$ attain their maximums on the boundary. Combining with  the Hopf lemma, we obtain
\begin{eqnarray}\label{mix-c1-b-est}
\sup_{z \in \partial M}|\nabla u|\leq C,
\end{eqnarray}
where $C$ depends on $(M, \omega), \chi_0, \alpha_l$ and $\underline{u}$.



\subsection{Notations and some lemmas}

In local complex coordinates $(z^1, ..., z^n)$, the subscripts of a function $u$ always denote the covariant derivatives of $u$
with respect to $\omega$ in the directions of the local frame $\frac{\partial }{\partial z^1}, ..., \frac{\partial}{\partial z^n}$.
Namely,
\begin{eqnarray*}
u_i=\nabla_{\frac{\partial}{\partial z^i}}u,
\quad u_{i\overline{j}}=\nabla_{\frac{\partial}{\partial \overline{z}^j}}
\nabla_{\frac{\partial}{\partial z^i}}u, \quad u_{i\overline{j}k}=\nabla_{\frac{\partial}{\partial z^k}}\nabla_{\frac{\partial}{\partial \overline{z}^j}}
\nabla_{\frac{\partial}{\partial z^i}}u.
\end{eqnarray*}
But, the covariant derivatives of a $(1, 1)$-form $\chi$
with respect to $\omega$ will be denoted by indices with
semicolons, e.g.,
\begin{eqnarray*}
\chi_{i\overline{j}; k}=\nabla_{\frac{\partial}{\partial z^k}}
\chi(\frac{\partial}{\partial z^i}, \frac{\partial}{\partial \overline{z}^j}), \quad
\chi_{i\overline{j}; k\overline{l}}=\nabla_{\frac{\partial}{\partial \overline{z}^{l}}}\nabla_{\frac{\partial}{\partial z^k}}
\chi(\frac{\partial}{\partial z^i}, \frac{\partial}{\partial \overline{z}^j}).
\end{eqnarray*}
We recall the following commutation formula on Hermitian manifolds $(M, \omega)$ \cite{Hou10, Guan10, Guan12}.

\begin{lemma}\label{2rd}
For $u \in C^4(M)$, we have
\begin{eqnarray*}
u_{i\overline{j}k}-u_{k\overline{j}i}=T_{ik}^{l}u_{l\overline{j}},
\end{eqnarray*}
\begin{eqnarray*}
u_{i\overline{j}\overline{k}}-u_{i\overline{k}\overline{j}}=\overline{T_{jk}^{l}}u_{i\overline{l}},
\end{eqnarray*}
\begin{eqnarray*}
u_{i\overline{j}k}=u_{ik \overline{j}}-R_{k\overline{j}i\overline{m}}g^{\overline{m} l}u_l,
\end{eqnarray*}
\begin{eqnarray*}
u_{i\overline{j}k\overline{l}}-u_{k\overline{l}i\overline{j}}
=g^{p\overline{q}}(R_{k\overline{l}i\overline{q}}u_{p\overline{j}}
-R_{i\overline{j}k\overline{q}}u_{p\overline{l}})+T_{ik}^{p}u_{p\overline{j}\overline{l}}
+\overline{T_{jl}^{q}}u_{i\overline{q}k}
-T_{ik}^{p}\overline{T_{jl}^{q}}u_{p\overline{q}},
\end{eqnarray*}
where $R$ is the curvature tensor of $(M, \omega)$.
\end{lemma}

\begin{theorem}\label{inter-c2}
Let $u\in C^{\infty}(\bar{M})$ be an $(k-1)$-admissible solution for equation \eqref{K-eq}.
Under the assumptions mentioned  in Theorem \ref{Main},  then there exists a positive constant $C$ depending only on  $(M, \omega), \chi_0$, $\alpha_l, \varphi$ and the subsolution $\underline{u}$ such that
 $$\sup_{ M} |\sqrt{-1}\partial \bar{\partial}u|\leq C\left(K+ \sup_{\partial M} |\sqrt{-1}\partial \bar{\partial}u|\right),$$
 where $K:=1+\sup_M|\nabla u|^2$, and $C$ depending on  $(M, \omega), \chi_0$, $\alpha_l, \varphi$ and the subsolution $\underline{u}$.
 \end{theorem}


\subsection{The proof of Theorem \ref{inter-c2}}

In order to study the  second order interior estimate of  \eqref{K-eq1} (or \eqref{K-eq}), it is sufficient to consider the following equation
\begin{eqnarray}\label{K-eq2}
\left\{
\begin{aligned}
& \frac{ \sigma_k(\widetilde{\chi}_{u})}{\sigma_{k-1}(\widetilde{\chi}_{u})}-\sum_{l=0}^{k-2} \beta_l(z) \frac{\sigma_l(\widetilde{\chi}_{u})}{\sigma_{k-1}(\widetilde{\chi}_{u})}=\beta(z), && in~ \bar{M},\\
&u=0 && on~ \partial M,
\end{aligned}
\right.
\end{eqnarray}
where $\widetilde{\chi}_{u}:=\widetilde{\chi}_0+u$ and $\widetilde{\chi}_0=\chi_0+\frac{\sqrt{-1}}{2}\partial\overline{\partial} \underline{u}$. It is clear that $0$ is an admissible subsolution to \eqref{K-eq2}. Without causing confusion, we denote $\widetilde{\chi}_{u}$ by $\chi$ and  $\widetilde{\chi}_0$ by $\chi_0$.

Following the work of Hou-Ma-Wu \cite{Hou10}, we define a function $W$ on $M$ as
\begin{eqnarray*}
W(z)=\log \lambda_1(z)+\varphi(|\nabla u|^2)+\psi(u),
\end{eqnarray*}
where $\lambda_1: M\rightarrow \mathbb{R}$ is the largest eigenvalue of the matrix
$A=\omega^{i\overline{l}}\chi_{\overline{l}j}$ at each point
with $\omega$. Since $\lambda_1$ is not a smooth function, we will perturb $A$ slightly as in \cite{Sze17, Sze18}.
There are also other methods to deal with this issue: one is to
use a viscosity type argument as in \cite{Tos21}, another is to replace $\lambda_1$ by a carefully chosen quadratic
function of $\chi_{i\overline{j}}$ as in \cite{Tos19-JRAM}.

The function $W$ must achieve its maximum at the interior point $z_0 \in M$. Around $z_0$, we choose a normal chart such that $A$ is
diagonal with eigenvalues
\begin{eqnarray*}
\lambda_1\geq ...\geq \lambda_n.
\end{eqnarray*}
We perturb $A$ by a diagonal matrix $B$ with $B_{11}=0$, small $0<B_{22}<...<B_{nn}$ and $B_{nn}<2B_{22}$.
Thus, the new matrix $\widetilde{A}=A-B$ has eigenvalues at $z_0$
\begin{eqnarray*}
\widetilde{\lambda}_1=\lambda_1, \quad \widetilde{\lambda}_i=\lambda_1-B_{ii} \quad \mbox{for} \quad i>1.
\end{eqnarray*}
Then, we can calculate the first and second derivatives of
$\widetilde{\lambda}_1$ at $z_0$ from Proposition \ref{mix-pro-54}
\begin{eqnarray*}
\widetilde{\lambda}_{1; i}=\widetilde{\lambda}_{1}^{pq}(\widetilde{A}_{pq})_{i}=\chi_{1\overline{1}; i}-(B_{11})_i,
\end{eqnarray*}
\begin{eqnarray*}
\widetilde{\lambda}_{1; i\overline{i}}&=&\widetilde{\lambda}_{1}^{pq, rs}
(\widetilde{A}_{pq})_{i}(\widetilde{A}_{rs})_{\overline{i}}+\widetilde{\lambda}_{1}^{pq}
(\widetilde{A}_{pq})_{i\overline{i}}\\&=&\chi_{1\overline{1}; i\overline{i}}
+\sum_{p>1}\frac{|\chi_{p\overline{1}; i}|^2+|\chi_{1\overline{p}; i}|^2}{\lambda_1-\widetilde{\lambda}_{p}}
\\&&+(B_{11})_{i\overline{i}}-2\mathrm{Re}
\sum_{p>1}\frac{\chi_{p\overline{1}; i}(B_{1\overline{p}})_{\overline{i}}+\chi_{1\overline{p}; i}(B_{p\overline{1}})_{\overline{i}}}{\lambda_1-\widetilde{\lambda}_{p}}\\&&+\widetilde{\lambda}_{1}^{pq, rs}
(B_{pq})_{i}(B_{rs})_{\overline{i}}.
\end{eqnarray*}
Since $\sum_{i}\lambda_i>0$, we can choose $B$ sufficient small such that $\sum_{i}\widetilde{\lambda}_i>0$.
Thus, $|\widetilde{\lambda}_i|<(n-1)\lambda_1$, which gives
\begin{eqnarray*}
\frac{1}{\lambda_1-\widetilde{\lambda}_i}>\frac{1}{n\lambda_1}.
\end{eqnarray*}
We can absorb the term $\chi_{p\overline{1};i}(B_{1\overline{p}})_{\overline{i}}$ using
\begin{eqnarray*}
|\chi_{p\overline{1};i}(B_{1\overline{p}})_{\overline{i}}|\leq \frac{1}{4}|\chi_{p\overline{1};i}|^2+C.
\end{eqnarray*}
Moreover, for $p>1$ there is
\begin{eqnarray*}
\frac{1}{\lambda_1-\widetilde{\lambda}_p}<\frac{1}{B_{22}}.
\end{eqnarray*}
It follows that
\begin{eqnarray}\label{Eigen-1}
\widetilde{\lambda}_{1; i\bar{i}}\geq \chi_{1\overline{1}; i\overline{i}}
+\frac{1}{2n\lambda_1}\sum_{p>1}(|\chi_{p\overline{1}; i}|^2+|\chi_{1\overline{p}; i}|^2)-C,
\end{eqnarray}
where $C$ depends on $|B|_{C^2(M)}$.

From Lemma \ref{2rd}, we have
\begin{eqnarray*}
u_{1\overline{1} i\overline{i}}=u_{i\overline{i}1\overline{1}}
-2\mathrm{Re}(u_{i\overline{p}1}\overline{T^{p}_{i1}})+\partial\overline{\partial} u*R+\partial\overline{\partial}u *T*T.
\end{eqnarray*}
where $*$ denotes a contraction. Then, we get
\begin{eqnarray}\label{C1-1}
\widetilde{\lambda}_{1;i\overline{i}}&\geq& u_{i\overline{i}1\overline{1}}
+\frac{1}{3n\lambda_1}\sum_{p>1}(|u_{p\overline{1} i}|^2+|u_{1\overline{p}i}|^2)
\\ \nonumber&&-2\mathrm{Re}(u_{i\overline{p}1}\overline{T^{p}_{i1}})-C\lambda_1-C.
\end{eqnarray}
Using $u_{i\overline{p}1}=u_{1\overline{p}i}+T_{ip}^{q}u_{q\overline{p}}$ and Cauchy-Schwartz inequality, we have
\begin{eqnarray}\label{C1-2}
2\sum_{p>1}|\mathrm{Re}(u_{i\overline{p}1}\overline{T^{p}_{i1}})|&\leq&
2\sum_{p>1}|\mathrm{Re}(u_{1\overline{p}i}\overline{T^{p}_{i1}})|+C\lambda_1 \nonumber \\&\leq&
\frac{1}{3n\lambda_1}\sum_{p>1}(|u_{p\overline{1}i}|^2+|u_{1\overline{p} i}|^2)+C\lambda_1.
\end{eqnarray}
Plugging \eqref{C1-2} into \eqref{C1-1}, we obtain
\begin{eqnarray}\label{C1-3}
\widetilde{\lambda}_{1;i\overline{i}}&\geq& u_{i\overline{i}1\overline{1}}
-2\mathrm{Re}(u_{i\overline{1}1}\overline{T^{1}_{i1}})-C\lambda_1
\nonumber\\&\geq& \chi_{i\overline{i};1\overline{1}}
-2\mathrm{Re}(\chi_{1\overline{1};i}\overline{T^{1}_{i1}})-C\lambda_1,
\end{eqnarray}
where we use $u_{i\overline{1}1}=u_{1\overline{1}i}+T_{i1}^{p}u_{p\overline{1}}$ and choose $\lambda_1$ large enough ($\lambda_1>>1$) to absorb a constant into $C\lambda_1$.

Differentiating the equation \eqref{K-eq2} twice, we have
\begin{eqnarray}\label{C1-mix-31}
\nabla_p \beta = G^{i\bar{j}} \chi_{i\bar{j};p}+\sum_{l=0}^{k-2} (\nabla_p \beta_l) G_l,
\end{eqnarray}
and
\begin{eqnarray}\label{C1-mix-32}
\nonumber\nabla_{\bar{p}}\nabla_p \beta &=& G^{i \bar{j}, r\bar{s}} \chi_{i\bar{j};p}  \chi_{r\bar{s};\bar{p}} +G^{i \bar{j}} \chi_{i\bar{j};p\bar{p}} +\sum_{l=0}^{k-2} \left ((\nabla_{\bar{p}} \beta_l) G^{i\bar{j}}_l \chi_{i\bar{j};p} + (\nabla_{p} \beta_l) G^{i\bar{j}}_l \chi_{i\bar{j};\bar{p}} \right)
\\&&+\sum_{l=0}^{k-2} (\nabla_{\bar{p}} \nabla_p \beta_l ) G_l.
\end{eqnarray}
Note that the operator $\left(\frac{\sigma_{k-1}}{\sigma_l}\right)^{\frac{1}{k-l-1}}$ and the operator $\frac{\sigma_k}{\sigma_{k-1}}$ is concave in $\Gamma_{k-1}$, we have
\begin{eqnarray}\label{C1-mix-33}
G_l^{i \bar{j}, r\bar{s}} \chi_{i\bar{j};p}  \chi_{r\bar{s};\bar{p}} \leq (1+\frac{1}{k-l-1}) G_l^{-1} G_l^{i\bar{j}} G_l^{r\bar{s}} \chi_{i\bar{j};p} \chi_{r\bar{s};\bar{p}}
\end{eqnarray}
and
\begin{eqnarray}\label{C1-mix-34}
G^{i\bar{j}, r \bar{s}}_k \chi_{i\bar{j};p}  \chi_{r\bar{s};\bar{p}} \leq 0.
\end{eqnarray}
Combining with \eqref{C1-mix-31}, \eqref{C1-mix-32}, \eqref{C1-mix-33} and \eqref{C1-mix-34}, we have
\begin{eqnarray}\label{C1-mix-35}
\nonumber&&\nabla_{\bar{p}}\nabla_p \beta-(1-\delta^2) G^{i\bar{j}, r \bar{s}} \chi_{i\bar{j};p}  \chi_{r\bar{s};\bar{p}}
\\\nonumber\leq  &&\sum_{l=0}^{k-2} \delta^2 \beta_l G^{i \bar{j}, r\bar{s}}_l \chi_{i\bar{j};p}  \chi_{r\bar{s};\bar{p}} +G^{i \bar{j}} \chi_{i\bar{j};p\bar{p}} +\sum_{l=0}^{k-2} \left ((\nabla_{\bar{p}} \beta_l) G^{i\bar{j}}_l \chi_{i\bar{j};p} + (\nabla_{p} \beta_l) G^{i\bar{j}}_l \chi_{i\bar{j};\bar{p}} \right)
\\\nonumber&&+\sum_{l=0}^{k-2} (\nabla_{\bar{p}} \nabla_p \beta_l ) G_l
\\\nonumber \leq && \delta^2 \sum_{l=0}^{k-2} \beta_l (1+\frac{1}{k-l-1}) G_l^{-1} |G_l^{i\bar{j}} \chi_{i\bar{j};p}|^2 +G^{i \bar{j}} \chi_{i\bar{j};p\bar{p}}
\\\nonumber&&+\sum_{l=0}^{k-2} \left ((\nabla_{\bar{p}} \beta_l) G^{i\bar{j}}_l \chi_{i\bar{j};p} + (\nabla_{p} \beta_l) G^{i\bar{j}}_l \chi_{i\bar{j};\bar{p}} \right)+\sum_{l=0}^{k-2} (\nabla_{\bar{p}} \nabla_p \beta_l ) G_l
\\\nonumber = && \frac{\delta^2 (k-l)}{k-1-l} \sum_{l=0}^{k-2} \beta_l G_l^{-1} \left|G_l^{i\bar{j}} \chi_{i\bar{j};p}+ \frac{1}{1+\frac{1}{k-l-1}} \frac{\nabla_p \beta_l}{\delta^2 \beta_l} G_l\right|^2
\\\nonumber&&- \sum_{l=0}^{k-2} \frac{k-l-1}{k-l} \frac{|\nabla_p \beta_l|^2}{\delta^2 \beta_l} G_l +G^{i \bar{j}} \chi_{i\bar{j};p\bar{p}} +\sum_{l=0}^{k-2} (\nabla_{\bar{p}} \nabla_p \beta_l ) G_l.
\end{eqnarray}
Therefore,
\begin{eqnarray}\label{C1-mix-36}
\nonumber G^{i \bar{j}} \chi_{i\bar{j};p\bar{p}}&\geq&\nabla_{\bar{p}}\nabla_p \beta-(1-\delta^2) G^{i\bar{j}, r \bar{s}} \chi_{i\bar{j};p}  \chi_{r\bar{s};\bar{p}}-\sum_{l=0}^{k-2} (\nabla_{\bar{p}} \nabla_p \beta_l ) G_l
\\ &&+\sum_{l=0}^{k-2} \frac{k-l-1}{k-l} \frac{|\nabla_p \beta_l|^2}{\delta^2 \beta_l} G_l,
\end{eqnarray}
which implies that
\begin{eqnarray}\label{C1-mix-37}
\nonumber G^{i\overline{i}}(\log \widetilde{\lambda}_1)_{i\overline{i}}&=& G^{i\overline{i}} \frac{\widetilde{\lambda}_{1; i\overline{i}}}{\lambda_1}- G^{i\overline{i}} \frac{|\widetilde{\lambda}_{1; i}|^2}{\lambda_1^2}
\\\nonumber&\geq& \frac{1}{\lambda_1} G^{i\overline{i}} \chi_{i\bar{i};1\bar{1}}
-\frac{2}{\lambda_1}G^{i\overline{i}} \mathrm{Re}(\chi_{1\overline{1};i}\overline{T^{1}_{i1}})-C\sum_i G^{i\overline{i}}-\frac{1}{\lambda_1^2} G^{i\overline{i}}|\chi_{1\bar{1};i}-B_{1\bar{1};i}|^2
\\\nonumber&\geq& \frac{1}{\lambda_1} \nabla_{\bar{1}}\nabla_1 \beta- \frac{1-\delta^2}{\lambda_1}  G^{i\bar{j}, r \bar{s}} \chi_{i\bar{j};1}  \chi_{r\bar{s};\bar{1}}- \frac{1}{\lambda_1}\sum_{l=0}^{k-2} (\nabla_{\bar{1}} \nabla_1 \beta_l ) G_l
\\ &&+ \frac{1}{\lambda_1} \sum_{l=0}^{k-2} \frac{k-l-1}{k-l} \frac{|\nabla_1 \beta_l|^2}{\delta^2 \beta_l} G_l -(1+\delta^4) G^{i\overline{i}} \frac{|\chi_{1\overline{1};i}|^2}{\lambda_1^2}-C\sum_i  G^{i\overline{i}}
\\\nonumber&\geq& - \frac{1-\delta^2}{\lambda_1}  G^{i\bar{j}, r \bar{s}} \chi_{i\bar{j};1}  \chi_{r\bar{s};\bar{1}} -(1+\delta^4) G^{i\overline{i}} \frac{|\chi_{1\overline{1};i}|^2}{\lambda_1^2}-C\sum_i  G^{i\overline{i}}-C.
\end{eqnarray}

Now we begin to prove Theorem \ref{inter-c2}. We redefine the auxiliary function
\begin{eqnarray*}
W(z)=\log \widetilde{\lambda}_1+\varphi(|\nabla u|^2)+\psi(u),
\end{eqnarray*}
where
\begin{eqnarray*}
\varphi(s)=-\frac{1}{2}\log \Big(1-\frac{s}{2K}\Big) \quad \mbox{for} \quad 0\leq s\leq K-1
\end{eqnarray*}
and
\begin{eqnarray*}
\psi(t)=-A\log \Big(1+\frac{t}{2L}\Big) \quad \mbox{for} \quad -L+1\leq t\leq 0.
\end{eqnarray*}
Here, we set
\begin{eqnarray*}
K=:\sup_{M}|\nabla u|^2+1, \quad L:=\sup_{M}|u|+1, \quad  A:=2L\Lambda,
\end{eqnarray*}
and $\Lambda$ is a large constant that which will be chosen later. Clearly, $\varphi$ satisfies
\begin{eqnarray*}
\frac{1}{2K}\geq \varphi^{\prime}\geq \frac{1}{4K}, \quad  \varphi^{\prime \prime}=2(\varphi^{\prime})^2>0,
\end{eqnarray*}
and $\psi$ satisfies
\begin{eqnarray*}
2\Lambda\geq -\psi^{\prime}\geq \Lambda, \quad  \psi^{\prime \prime}
\geq\frac{2\varepsilon}{1-\varepsilon}(\psi^{\prime})^2 \quad \mbox{for all} \quad \varepsilon\leq\frac{1}{2A+1}.
\end{eqnarray*}
The function $W$ must achieve its maximum at the interior point $z\in M$. Thus, we arrive at $z$
\begin{eqnarray}\label{Diff-1}
W_i=\frac{\chi_{1\overline{1};i}}{\chi_{1\overline{1}}}
+\varphi^{\prime}\nabla_i(|\nabla u|^2)+\psi^{\prime}u_i=0
\end{eqnarray}
and
\begin{eqnarray}\label{Diff-2}
W_{i\overline{i}}&=&(\log \widetilde{\lambda}_{1})_{ii}
+\varphi^{\prime\prime}|\nabla_i(|\nabla u|^2)|^2
+\varphi^{\prime}\nabla_{\overline{i}}\nabla_i(|\nabla u|^2)\nonumber\\&&+\psi^{\prime\prime}|u_i|^2+\psi^{\prime}u_{i\overline{i}}\leq 0.
\end{eqnarray}

Multiplying \eqref{Diff-2} by $G^{i\overline{i}}$ and summing it over index $i$, we can know from \eqref{C1-mix-37}
\begin{eqnarray}\label{C2-3}
0&\geq& - \frac{1-\delta^2}{\chi_{1\bar{1}}}  G^{i\bar{j}, r \bar{s}} \chi_{i\bar{j};1}  \chi_{r\bar{s};\bar{1}} -(1+\delta^4) G^{i\overline{i}} \frac{|\chi_{1\overline{1};i}|^2}{\chi_{1\bar{1}}^2}-C\sum_i  G^{i\overline{i}}-C
\nonumber \\&&+\varphi^{\prime\prime}G^{i\overline{i}}|\nabla_i(|\nabla u|^2)|^2
+\varphi^{\prime}G^{i\overline{i}}\nabla_{\overline{i}}\nabla_{i}(|\nabla u|^2)
+\psi^{\prime\prime}G^{i\overline{i}}|u_i|^2+\psi^{\prime}G^{i\overline{i}}u_{i\overline{i}}.
\end{eqnarray}
To proceed, we need the following calculation
\begin{eqnarray*}
\nabla_i(|\nabla u|^2)=\sum_{p}(u_p u_{\overline{p}i}+u_{\overline{p}}u_{pi})
\end{eqnarray*}
and
\begin{eqnarray*}
\nabla_{\overline{i}}\nabla_i(|\nabla u|^2)=
\sum_{p}(u_{p\overline{i}} u_{\overline{p}i}+u_{\overline{p}\overline{i}}u_{pi}
+u_p u_{\overline{p}i\overline{i}}+u_{\overline{p}}u_{pi\overline{i}}).
\end{eqnarray*}
Note that
\begin{eqnarray*}
u_{\overline{p}i\overline{i}}=u_{i\overline{p}\overline{i}}=u_{i\overline{i}\overline{p}}
+\overline{T_{pi}^{q}}u_{i\overline{q}}=\chi_{i\overline{i}; \overline{p}}
-(\chi_0)_{i\overline{i};\overline{p}}
+\overline{T_{pi}^{i}}\chi_{i\bar{i}}-\overline{T_{pi}^{q}}(\chi_0)_{i\overline{q}}
\end{eqnarray*}
and
\begin{eqnarray*}
u_{pi\overline{i}}&=&u_{p\overline{i}i}+R_{i\overline{i}p\overline{q}}g^{\overline{q}m}u_m
\\&=&u_{i\overline{i}p}+R_{i\overline{i}p\overline{q}}g^{\overline{q}m}u_m+T_{pi}^{q}u_{q\overline{i}}
\\&=&\chi_{i\overline{i};p}-(\chi_0)_{i\overline{i};p}+R_{i\overline{i}p\overline{q}}g^{\overline{q}m}u_m+T_{pi}^{i}\chi_{i\bar{i}}
-T_{pi}^{q}(\chi_0)_{q\overline{i}}.
\end{eqnarray*}
Combining with \eqref{C1-mix-31},   we have
\begin{eqnarray*}
G^{i\overline{i}}u_{\overline{p}i\overline{i}}u_p&\geq&G^{i\overline{i}}\chi_{i\overline{i};\overline{p}}u_p
+G^{i\overline{i}}\overline{T_{pi}^{i}}\chi_{i\bar{i}}u_p-C_{\epsilon}K\sum_{i}G^{i\overline{i}}
\\&\geq& (\beta )_{\overline{p}}u_p-
\frac{\epsilon}{2} G^{i\overline{i}} |\chi_{i\bar{i}}|^{2}-C_{\epsilon}K\sum_{i}G^{i\overline{i}}.
\end{eqnarray*}
Similarly,
\begin{eqnarray*}
G^{i\overline{i}}u_{pi\overline{i}}u_{\overline{p}}\geq(\beta)_{p}u_{\overline{p}}-\frac{\epsilon}{2} G^{i\overline{i}}|\chi_{i\bar{i}}|^{2}-C_{\epsilon}K\sum_{i}G^{i\overline{i}}.
\end{eqnarray*}
Thus,
\begin{eqnarray}\label{C2-3-1}
G^{i\overline{i}}\nabla_{\overline{i}}\nabla_i(|\nabla u|^2)
&\geq& \sum_{p}G^{i\overline{i}}(|u_{pi}|^2+|u_{\overline{p}i}|^2)+2\sum_{p} \mathrm{Re}\{(\beta)_{p}u_{\overline{p}}\}\nonumber\\&&-
\epsilon G^{i\overline{i}}|\chi_{i\bar{i}}|^{2}-C_{\epsilon}K\sum_{i}G^{i\overline{i}}.
\end{eqnarray}
Since
\begin{eqnarray}\label{C8-1}
\sum_{p}G^{i\overline{i}}|u_{\overline{p}i}|^2
&\geq& G^{i\overline{i}}|u_{\overline{i}i}|^2\geq G^{i\overline{i}}|\chi_{i\bar{i}}-(\chi_{0})_{i\overline{i}}|^2
\nonumber\\&\geq& \frac{1}{2}G^{i\overline{i}}|\chi_{i\bar{i}}|^{2}-C\sum_{i}G^{i\overline{i}},
\end{eqnarray}
we can use half of the term $\sum_{p}G^{i\overline{i}}|u_{\overline{p}i}|^2$ to absorb the negative term $-\epsilon G^{i\overline{i}}|\chi_{i\bar{i}}|^{2}$ if we chose $\epsilon=\frac{1}{8}$.
From now on we can replace  $C_{\epsilon}$ with $C$ since $\epsilon$ is fixed. It follows from \eqref{C2-3-1}
\begin{eqnarray}\label{C2-8-1}
G^{i\overline{i}}\nabla_{\overline{i}}\nabla_i(|\nabla u|^2)
&\geq& \frac{1}{8}G^{i\overline{i}}|\chi_{i\bar{i}}|^{2}+\frac{1}{2}\sum_{p}G^{i\overline{i}}(|u_{pi}|^2+|u_{\overline{p}i}|^2)
\nonumber\\&&+2\sum_{p} \mathrm{Re}\{(\beta)_{p}u_{\overline{p}}\}
-CK\sum_{i}G^{i\overline{i}}.
\end{eqnarray}
Taking the inequality \eqref{C2-8-1} into \eqref{C2-3}, it yields

\begin{eqnarray}\label{C2-4}
0&\geq& - \frac{1-\delta^2}{\chi_{1\bar{1}}}  G^{i\bar{j}, r \bar{s}} \chi_{i\bar{j};1}  \chi_{r\bar{s};\bar{1}} -(1+\delta^4) G^{i\overline{i}} \frac{|\chi_{1\overline{1};i}|^2}{\chi_{1\bar{1}}^2}+\psi^{\prime\prime}G^{i\overline{i}}|u_i|^2+\psi^{\prime}G^{i\overline{i}}u_{i\overline{i}}
\nonumber \\&&+\varphi^{\prime\prime}G^{i\overline{i}}|\nabla_i(|\nabla u|^2)|^2
+\frac{\varphi^{\prime}}{8}G^{i\overline{i}}|\chi_{i\bar{i}}|^{2}+\frac{\varphi^{\prime}}{2}\sum_{p}G^{i\overline{i}}(|u_{pi}|^2+|u_{\overline{p}i}|^2)
\nonumber\\&&+ 2 \varphi^{\prime} \sum_{p} \mathrm{Re}\{(\beta)_{p}u_{\overline{p}}\}-C\sum_i  G^{i\overline{i}}-C.
\end{eqnarray}
Now, we divide our proof into two cases separately, depending on whether
$\chi_{n\overline{n}}<-\delta\chi_{1\overline{1}}$ or not, for a small $\delta$ to be chosen later.

\textbf{Case 1.} $\chi_{n\overline{n}}<-\delta\chi_{1\overline{1}}$. In this case, it follows that
$\chi_{1\overline{1}}^{2}\leq \frac{1}{\delta^2}\chi_{n\overline{n}}^{2}$.
So, we only need to bound $\chi_{n\overline{n}}^{2}$. Clearly, we can obtain
if we throw some positive terms in \eqref{C2-4}
\begin{eqnarray}\label{C2-5}
0&\geq&  -(1+\delta^4) G^{i\overline{i}} \frac{|\chi_{1\overline{1};i}|^2}{\chi_{1\bar{1}}^2}+\psi^{\prime}G^{i\overline{i}}u_{i\overline{i}}+\varphi^{\prime\prime}G^{i\overline{i}}|\nabla_i(|\nabla u|^2)|^2
+\frac{\varphi^{\prime}}{8}G^{i\overline{i}}|\chi_{i\bar{i}}|^{2}
\nonumber\\&&+ 2 \varphi^{\prime} \sum_{p} \mathrm{Re}\{(\beta)_{p}u_{\overline{p}}\}-C \sum_i  G^{i\overline{i}}-C.
\end{eqnarray}
From \eqref{ineq-mix-02}, we obtain
\begin{eqnarray}\label{C2-6}
-\psi^{\prime}\sum_{i}G^{i\overline{i}}u_{i\overline{i}}&=&-\psi^{\prime}\sum_{i}G^{i\overline{i}}
\Big[\chi_{i\overline{i}}-((\chi_0)_{i\overline{i}}-\tau)-\tau\Big]\nonumber
\\&=&-\psi^{\prime}\Big[\beta+\sum_{l=0}^{k-2} (l-k) \beta_l G_l -\sum_{i}G^{i\overline{i}}
(\chi_{0i\overline{i}}-\tau)-\tau\sum_{i}G^{i\overline{i}}\Big]\nonumber\\&\leq&
C\Lambda-\tau\Lambda\sum_{i}G^{i\overline{i}}.
\end{eqnarray}
Plugging \eqref{C2-6} into \eqref{C2-5} and choosing $\Lambda$ large enough, 
\begin{eqnarray}\label{C2-7}
&&\varphi^{\prime\prime}G^{i\overline{i}}|\nabla_i(|\nabla u|^2)|^2
+\frac{\varphi^{\prime}}{8}G^{i\overline{i}}|\chi_{i\bar{i}}|^{2}\\
\nonumber&\leq&  (1+\delta^4) G^{i\overline{i}} \frac{|\chi_{1\overline{1};i}|^2}{\chi_{1\bar{1}}^2}- 2 \varphi^{\prime} \sum_{p} \mathrm{Re}\{(\beta)_{p}u_{\overline{p}}\}+C(1+\Lambda).
\end{eqnarray}
Note that we get from \eqref{Diff-1}
\begin{eqnarray}\label{C2-8}
(1+\delta^4)\sum_{i}\frac{G^{i\overline{i}}|\chi_{1\overline{1};i}|^2}{\chi_{1\overline{1}}^2}
&=& (1+\delta^4)\sum_{i}G^{i\overline{i}}|\varphi^{\prime}\nabla_i(|\nabla u|^2)+\psi^{\prime}u_i|^2
\nonumber\\&\leq&2(\varphi^{\prime})^2\sum_{i}G^{i\overline{i}}|\nabla_i(|\nabla u|^2)|^2
+\frac{8(1+\delta^4)}{1-\delta^4}\Lambda^2K\sum_{i}G^{i\overline{i}}
\nonumber\\&\leq&2(\varphi^{\prime})^2\sum_{i}G^{i\overline{i}}|\nabla_i(|\nabla u|^2)|^2
+16\Lambda^2K\sum_{i}G^{i\overline{i}},
\end{eqnarray}
where we choose $\delta\leq 3^{-\frac{1}{4}}$. In fact,
\begin{eqnarray}\label{C2-91}
\sum_{i}G^{i\overline{i}}\chi_{i\overline{i}}^2\geq G^{n\overline{n}}\chi_{n\overline{n}}^2
\geq\frac{1}{n}\chi_{n\overline{n}}^2\sum_{i} G^{i\overline{i}}.
\end{eqnarray}
Substituting \eqref{C2-8} and \eqref{C2-91} into \eqref{C2-7}, we get
\begin{eqnarray*}
\frac{1}{32nK}\chi_{n\overline{n}}^2\sum_{i} G^{i\overline{i}}\leq 16\Lambda^2K\sum_{i}G^{i\overline{i}}+C(1+\Lambda).
\end{eqnarray*}
Combining with \eqref{hq-ineq-3}, it is easy to derive that
\begin{eqnarray*}
\chi_{1\overline{1}}\leq C K.
\end{eqnarray*}

\textbf{Case 2. }$\chi_{n\overline{n}}\geq-\delta \chi_{1\overline{1}}$.
Define
\begin{eqnarray*}
I=\Big\{i \in \{1, 2, ..., n\}: G^{i\overline{i}}>\delta^{-1}G^{1\overline{1}}\Big\}.
\end{eqnarray*}
It follows from \eqref{f-2}
\begin{eqnarray}\label{C2-9}
-\frac{1}{\chi_{1\overline{1}}}G^{i\overline{j},r\overline{s}}\chi_{i\overline{j};
1}\chi_{r\overline{s}; \overline{1}}&\geq&
\frac{1-\delta}{1+\delta}\frac{1}{\chi_{1\overline{1}}^2}\sum_{i \in I}G^{i\overline{i}}|\chi_{i\overline{1};
1}|^2\nonumber\\&\geq&\frac{1-\delta}{1+\delta}\frac{1}{\chi_{1\overline{1}}^2}\sum_{i \in I}G^{i\overline{i}}\Big(|\chi_{1\overline{1};
i}|^2+2\mathrm{Re}\{\chi_{1\overline{1};
i}\overline{e'_i}\}\Big),
\end{eqnarray}
where $e'_i=T_{i1}^{p}u_{p\overline{1}}+(\chi_0)_{i \overline{1}; 1}-(\chi_0)_{\overline{1}1; i}$.

Note that $\varphi^{\prime \prime}=2(\varphi^{\prime})^2$, using \eqref{Diff-1}, we get
\begin{eqnarray}\label{C2-10}
&&\varphi^{\prime \prime}\sum_{i \in I}G^{i\overline{i}}|\nabla_i(|\nabla u|^2)|^2
\geq2\sum_{i \in I}G^{i\overline{i}}\bigg(\delta \Big|\frac{\chi_{1\overline{1};
i}}{\chi_{1\overline{1}}}\Big|^2-\frac{\delta}{1-\delta}|\psi^{\prime}u_i|^2\bigg).
\end{eqnarray}
Choosing $\delta\leq\min\{\frac{1}{2A+1}, 3^{-\frac{1}{4}}\}$ and hence $\psi^{\prime \prime}
\geq\frac{2\delta}{1-\delta}(\psi^{\prime})^2$. Combining \eqref{C2-9} with \eqref{C2-10},
\begin{eqnarray}\label{C2-11}
&&-(1+\delta^4)\sum_{i \in I}\frac{G^{i\overline{i}}|\chi_{1\overline{1};i}|^2}{\chi_{1\overline{1}}^2}
-\frac{1-\delta^2}{\chi_{1\overline{1}}}G^{i\overline{j},r\overline{s}}\chi_{i\overline{j};
1}\chi_{r\overline{s}; \overline{1}}\nonumber\\&& \quad
+\varphi^{\prime\prime}\sum_{i \in I}G^{i\overline{i}}|\nabla_i(|\nabla u|^2)|^2
+\psi^{\prime\prime}\sum_{i \in I}G^{i\overline{i}}|u_i|^2\nonumber\\&\geq&\frac{\delta^2}{2}\sum_{i \in I}\frac{G^{i\overline{i}}|\chi_{1\overline{1};i}|^2}{\chi_{1\overline{1}}^2}
+2(1-\delta)^2\frac{1}{\chi_{1\overline{1}}^2}\sum_{i \in I}G^{i\overline{i}}\mathrm{Re}\{\chi_{1\overline{1};
i}\overline{e'_i}\}\nonumber\\&\geq&\frac{\delta^2}{4}\sum_{i \in I}\frac{G^{i\overline{i}}|\chi_{1\overline{1};i}|^2}{\chi_{1\overline{1}}^2}
-C_{\delta}\frac{1}{\chi_{1\overline{1}}^2}\sum_{i \in I}G^{i\overline{i}}\lambda_{i}^{2}
\nonumber\\&\geq&\frac{\delta^2}{4}\sum_{i \in I}\frac{G^{i\overline{i}}|\chi_{1\overline{1};i}|^2}{\chi_{1\overline{1}}^2}
-C\sum_{i \in I}G^{i\overline{i}}
,
\end{eqnarray}
where we choose $\chi_{1\overline{1}}$ large enough to get the last inequality.

For the terms without an index in $I$,  by \eqref{C2-8} and the fact $1 \notin I$ , it follows that
\begin{eqnarray}\label{C2-12}
&&-(1+\delta^4)\sum_{i \notin I}\frac{G^{i\overline{i}}|\chi_{1\overline{1};i}|^2}{\chi_{1\overline{1}}^2}
+\varphi^{\prime\prime}\sum_{i \notin I}G^{i\overline{i}}|\nabla_i(|\nabla u|^2)|^2
\\\nonumber&\geq&-\frac{16n\Lambda^2K}{\delta}G^{1\overline{1}}.
\end{eqnarray}
Substituting \eqref{C2-11} and \eqref{C2-12} into \eqref{C2-4},
\begin{eqnarray}\label{C2-151-1}
C\sum_i  G^{i\overline{i}}+C+\frac{16n\Lambda^2K}{\delta}G^{1\overline{1}}&\geq&
\frac{1}{8K}G^{i\overline{i}}\chi^{2}_{i\bar{i}}+\psi^{\prime}G^{i\overline{i}} u_{i\overline{i}}.
\end{eqnarray}
 Note that  $\mu:= \lambda (\chi_0+ \frac{\sqrt{-1}}{2} \partial \bar{\partial} \underline{u})$ satisties the condition \eqref{lem-mix-71}  since \eqref{sub-condition}. Without loss of generality we assume that $\chi_{1\bar{1}} \geq N$. From Lemma \ref{lemm-mix-23},  we will divide the following argument into two case:

 \textbf{Case 2.1:}
\begin{equation*}
\sum_i G^{i\overline{i}} u_{i\overline{i}} \leq -\theta-\theta \sum_i G^{i\overline{i}}.
\end{equation*}
It implies that
\begin{equation*}
\psi^{\prime}\sum_i G^{i\overline{i}} u_{i\overline{i}} \geq \Lambda\theta(1+ \sum_i G^{i\overline{i}}).
\end{equation*}
Combining with \eqref{C2-151-1}, we have
\begin{eqnarray}\label{C2-1512}
C\sum_i  G^{i\overline{i}}+C+\frac{16n\Lambda^2K}{\delta}G^{1\overline{1}}&\geq&
\frac{1}{8K}G^{i\overline{i}}\chi^{2}_{i\bar{i}}+\Lambda\theta(1+ \sum_i G^{i\overline{i}}).
\end{eqnarray}
We can  choose $\Lambda$ large enough, then \eqref{C2-1512} gives
$$\frac{1}{8K}G^{1\overline{1}}\chi^{2}_{1\bar{1}} \leq\frac{1}{8K}G^{i\overline{i}}\chi^{2}_{i\bar{i}} \leq \frac{16n\Lambda^2K}{\delta}G^{1\overline{1}}.$$
So
$$|\chi_{1\bar{1}}| \leq 8 \Lambda K  \sqrt{\frac{2n}{\delta}}.$$

\textbf{Case 2.2:}
\begin{equation}\label{C2000}
G^{1\bar{1}} \chi_{1\bar{1}} \geq \theta.
\end{equation}

From \eqref{C2-6}, we can absorb the last  term of \eqref{C2-151-1} into $\tau\Lambda\sum_{i}G^{i\overline{i}}$,
\begin{eqnarray*}
C(1+\Lambda)+C\sum_{i}G^{i\overline{i}}+\frac{16n\Lambda^2K}{\delta}G^{1\overline{1}}
\geq\frac{1}{8K}G^{1\overline{1}}\chi_{1\overline{1}}^2+\tau\Lambda\sum_{i}G^{i\overline{i}}.
\end{eqnarray*}
Choosing $\Lambda$ large enough, it follows that
\begin{eqnarray}\label{C2-16}
C(1+\Lambda)+\frac{16n\Lambda^2K}{\delta}G^{1\overline{1}}
\geq\frac{1}{8K}G^{1\overline{1}}\chi_{1\overline{1}}^2.
\end{eqnarray}
Using \eqref{C2000}, we choose $\chi_{1\overline{1}}$ large enough such that
\begin{eqnarray*}
\frac{1}{16K}G^{1\overline{1}}\chi_{1\overline{1}}^2\geq \frac{\theta}{16K} \chi_{1\overline{1}} \geq C(1+\Lambda).
\end{eqnarray*}
Taking the above inequality into \eqref{C2-16}, we arrive
\begin{eqnarray*}\label{C2-15}
\chi_{1\overline{1}}\leq 16 \sqrt{\frac{n}{\delta}}\Lambda K.
\end{eqnarray*}
So, we complete the proof.

\section{Boundary mixed tangential-normal estimates}


The goal of this section is to prove an estimate for the boundary mixed tangential-normal  derivatives. For $p \in \partial M$, we choose a coordinate $z=(z^1, ..., z^n)$ such that $p$ corresponds to the origin and
$g_{i\overline{j}}=\delta_{ij}$. We denote $z^i=x^i+\sqrt{-1}y^i$ and
\begin{eqnarray*}
t^{1}=y^1, t^2=y^2, ..., t^n=y^n, t^{n+1}=x^1, ..., t^{2n-1}=x^{n-1}.
\end{eqnarray*}
Let $\rho$ be
the distance function to $0$, i.e.,
\begin{eqnarray*}
\rho(z)=:dist_{g_0}(z, 0),  \quad  x\in M
\end{eqnarray*}
and set
$$M_{\delta}=\{z\in M: \rho(z)<\delta\}\quad \mbox{for} \quad
\delta>0.$$

 Let $d$ be the distance function to the boundary
$\partial M$ with respect to the background metric $g$
\begin{eqnarray*}
d(z)=:dist_{g_0}(z, \partial M),  \quad  z\in M.
\end{eqnarray*}
Since $\partial M$ is smooth and $|\nabla d|=1$ on $\partial M$,
we can choose $\delta>0$ sufficiently small so that $d$ is
smooth,
\begin{eqnarray}\label{d-1}
\frac{1}{2}\leq |\nabla d|\leq 2 \quad \mbox{in} \quad M_{\delta}
\end{eqnarray}
and
\begin{eqnarray}\label{d-2}
C_2\leq |\nabla^2 d|\leq C_1 \quad \mbox{in} \quad M_{\delta},
\end{eqnarray}
where the constants $C_1$ and $C_2$ are independent of $\delta$.

Suppose near the origin, the boundary $\partial M$ is represented by
\begin{eqnarray*}
\rho(z)=0
\end{eqnarray*}
and $d \rho \neq 0 $ on $\partial M$.
Then, there exists a function $\zeta(t)$ such that
\begin{eqnarray*}
\rho(t, \zeta(t))=0.
\end{eqnarray*}
Differentiating the boundary condition $(u-\underline{u})(t, \zeta(t))=0$,
we derive the relation on $\partial M \cap \overline{M_{\delta}}$
\begin{eqnarray}\label{Bd-D}
\partial_{t^{\alpha}}(u-\underline{u})=-\partial_{x^n}(u-\underline{u})\partial_{t^{\alpha}}\zeta
\end{eqnarray}
for any $\alpha=1,\cdots, 2n-1$.
Moreover, differentiating the boundary condition again gives
\begin{eqnarray}\label{Bd-DD}
\partial_{t^{\alpha}}\partial_{t^{\gamma}}(u-\underline{u})(0)
=-\partial_{x^n}(u-\underline{u})(0)\partial_{t^{\alpha}}\partial_{t^{\gamma}}\zeta(0),
\end{eqnarray}
which implies that
\begin{eqnarray}\label{mix-ttest-1}
|\partial_{t^{\alpha}}\partial_{t^{\gamma}}u(0)|\leq C
\end{eqnarray}
for any $\alpha, \gamma=1,\cdots, 2n-1$.

To derive the boundary tangential-normal estimates, we use the following locally
defined auxiliary function in $M_{\delta}$
\begin{eqnarray*}
\Psi&=&A\sqrt{K}v+B\sqrt{K}|z|^2-\frac{1}{\sqrt{K}}\sum_{p=1}^{n}
\Big[\partial_{y^p}(u-\underline{u})\Big]^2
-\frac{1}{\sqrt{K}}\sum_{a=1}^{n-1}|\nabla_{a}(u-\underline{u})|^2\\&&+T_{\alpha}(u-\underline{u}),
\end{eqnarray*}
where $A, B\gg 1$ are constants to be determined. Here,
\begin{eqnarray*}
v=u-\underline{u}+t d-\frac{N}{2}d^2,
\end{eqnarray*}
introduced by  Guan in \cite{Guan14}, and $t$ and $N$ will be determined later.


\begin{lemma}
There exist uniform positive constants $t,  \delta, \varepsilon$ small enough and $N>>1$  such that $v$ astisfies
\begin{eqnarray}\label{v}
\left\{
\begin{aligned}
& G^{i\overline{j}}v_{i\overline{j}}\leq-\frac{\varepsilon }{4}\Big(1+\sum_{i=1}^{n}G^{i\overline{i}}\Big), && in~ M_{\delta},\\
&v\geq 0 && on~ \partial M_{\delta},
\end{aligned}
\right.
\end{eqnarray}
\end{lemma}
\begin{proof}
The proof is quite similar to that given for \cite[Lemma 2.4]{Zhang21} and so they are omitted.
\end{proof}

\subsection{Estimates of tangential derivatives $T_{\alpha}(u-\underline{u})$}

For $\alpha \in \{1, 2, ...,, 2n-1\}$, we define the real vector fields
\begin{eqnarray*}
T_{\alpha}=\frac{\partial}{\partial t^{\alpha}}-\frac{\rho_{t^{\alpha}}}
{\rho_{x^{n}}}\frac{\partial}{\partial x^{n}},
\end{eqnarray*}
which are clearly tangential vector on $\partial M$. Then, we have

\begin{lemma}
\begin{eqnarray}\label{bar2}
\left|G^{i\overline{j}}\partial_i\partial_{\overline{j}}T_{\alpha}(u-\underline{u})\right|&\leq&
\frac{1}{\sqrt{K}}G^{i\overline{j}}\partial_{i}\partial_{y^n}(u-\underline{u})
\partial_{\overline{j}}\partial_{y^n}(u-\underline{u})\nonumber\\&&+C \sum_{i=1}^n (1+G^{i\overline{i}}|\lambda_i|)+C(1+\sqrt{K})\sum_{i=1}^{n}G^{i\overline{i}}.
\end{eqnarray}
\end{lemma}

\begin{proof}
We start with the following computation
\begin{eqnarray}\label{B3}
&&G^{i\overline{j}}\partial_i\partial_{\overline{j}}T_{\alpha}(u-\underline{u})
\nonumber\\&=&G^{i\overline{j}}\partial_i\partial_{\overline{j}}\partial_{t^{\alpha}}(u-\underline{u})
-\frac{\rho_{t^{\alpha}}}{\rho_{x^n}}G^{i\overline{j}}\partial_i\partial_{\overline{j}}\partial_{x^n}(u-\underline{u})
\nonumber\\&&-2\mathrm{Re}\bigg[G^{i\overline{j}}\Big(\partial_i\frac{\rho_{t^{\alpha}}}{\rho_{x^n}}\Big)
\partial_{\overline{j}}\partial_{x^n}(u-\underline{u})\bigg]
-G^{i\overline{j}}\Big(\partial_i\partial_{\overline{j}}\frac{\rho_{t^{\alpha}}}{\rho_{x^n}}\Big)\partial_{x^n}(u-\underline{u}).
\end{eqnarray}
Differentiating the equation \eqref{K-eq1} once, we have
\begin{eqnarray}\label{B5}
G^{i\overline{j}}\nabla_{t^{\alpha}}\nabla_i\nabla_{\overline{j}}u
+G^{i\overline{j}}\nabla_{t^{\alpha}}(\chi_0)_{\overline{j}i}=\nabla_{t^{\alpha}}\beta- \sum_{l=0}^{k-2}  (\nabla_{t^{\alpha}}\beta_l) G_l.
\end{eqnarray}
Note that
\begin{eqnarray*}
\frac{\partial}{\partial t^{\alpha}}=\frac{1}{\sqrt{-1}}
\bigg(\frac{\partial}{\partial z^{\alpha}}-\frac{\partial}
{\partial \overline{\overline{z}}^{\alpha}}\bigg) \quad \mbox{for} \quad 1\leq\alpha\leq n,
\end{eqnarray*}
and
\begin{eqnarray*}
\frac{\partial}{\partial t^{\alpha}}=
\bigg(\frac{\partial}{\partial z^{\alpha-n}}+\frac{\partial}
{\partial \overline{\overline{z}}^{\alpha-n}}\bigg) \quad \mbox{for} \quad n+1\leq\alpha\leq 2n-1.
\end{eqnarray*}
Then, we can convert covariant derivatives to partial derivatives for $n+1\leq\alpha\leq 2n-1$
\begin{eqnarray}\label{B6}
\nabla_{t^{\alpha}}\nabla_i\nabla_{\overline{j}}u
&=&\nabla_{\alpha}\nabla_i\nabla_{\overline{j}}u+\nabla_{\overline{\alpha}}\nabla_i\nabla_{\overline{j}}u
\nonumber\\&=&\partial_i\partial_{\overline{j}} \partial_{t^{\alpha}}u-\Gamma_{\alpha i}^{r}u_{\overline{j}r}-\overline{\Gamma_{\alpha j}^{r}}u_{\overline{r}i}.
\end{eqnarray}
Substituting \eqref{B6} into \eqref{B5} gives
\begin{eqnarray}\label{B7}
G^{i\overline{j}}\partial_i\partial_{\overline{j}}\partial_{t^{\alpha}}u
&=&-G^{i\overline{j}}\nabla_{t^{\alpha}}(\chi_0)_{\overline{j}i}+\nabla_{t^{\alpha}}\beta- \sum_{l=0}^{k-2}  (\nabla_{t^{\alpha}}\beta_l) G_l
\nonumber\\&&+G^{i\overline{j}}\Gamma_{\alpha i}^{r}u_{\overline{j}r}+G^{i\overline{j}}\overline{\Gamma_{\alpha j}^{r}}u_{\overline{r}i}.
\end{eqnarray}
It follows that for $n+1\leq\alpha\leq 2n-1$
\begin{eqnarray}\label{B7}
\left|G^{i\overline{j}}\partial_i\partial_{\overline{j}}\partial_{t^{\alpha}}u\right|\leq CG^{i\overline{i}}(1+|\lambda_i|)
+C(|\nabla \beta|+\sum_{l=0}^{k-2}|\nabla \beta_l|).
\end{eqnarray}
Similarly, we can get the estimate \eqref{B7} for $1\leq\alpha\leq n$.
Thus, the estimate \eqref{B7} holds for all $1\leq\alpha\leq 2n-1$. Moreover, we also have a similar estimate
\begin{eqnarray}\label{B8}
|G^{i\overline{j}}\partial_i\partial_{\overline{j}}\partial_{x^{n}}u|\leq CG^{i\overline{i}}(1+|\lambda_i|)
+C(|\nabla \beta|+\sum_{l=0}^{k-2}|\nabla \beta_l|).
\end{eqnarray}
Plugging \eqref{B7} and \eqref{B8} into \eqref{B3}, it yields
\begin{eqnarray}\label{B9}
G^{i\overline{j}}\partial_i\partial_{\overline{j}}T_{\alpha}(u-\underline{u})&\leq&
-2\mathrm{Re}\bigg[G^{i\overline{j}}\Big(\partial_i\frac{\rho_{t^{\alpha}}}{\rho_{x^n}}\Big)
\partial_{\overline{j}}\partial_{x^n}(u-\underline{u})\bigg]
+\sqrt{K}\sum_{i=1}^{n}G^{i\overline{i}}\nonumber\\&&+CG^{i\overline{i}}(1+|\lambda_i|)+C.
\end{eqnarray}
The first term on the right hand side of \eqref{B9} can be written as
\begin{eqnarray*}
G^{i\overline{j}}\Big(\partial_i\frac{\rho_{t^{\alpha}}}{\rho_{x^n}}\Big)
\partial_{\overline{j}}\partial_{x^n}(u-\underline{u})&=&2G^{i\overline{j}}\Big(\partial_i\frac{\rho_{t^{\alpha}}}{\rho_{x^n}}\Big)
\partial_{\overline{j}}\partial_{n}(u-\underline{u})\nonumber\\&&+\sqrt{-1}G^{i\overline{j}}\Big(\partial_i\frac{\rho_{t^{\alpha}}}{\rho_{x^n}}\Big)
\partial_{\overline{j}}\partial_{y^n}(u-\underline{u}),
\end{eqnarray*}
it follows from  the Cauchy-Schwarz inequality that
\begin{eqnarray}\label{B10}
\nonumber|G^{i\overline{j}}\Big(\partial_i\frac{\rho_{t^{\alpha}}}{\rho_{x^n}}\Big)
\partial_{\overline{j}}\partial_{x^n}(u-\underline{u})|&\leq&\frac{1}{2\sqrt{K}}G^{i\overline{j}}\partial_{i}\partial_{y^n}(u-\underline{u})
\partial_{\overline{j}}\partial_{y^n}(u-\underline{u})\\&&+C (\sum_{i=1}^{n}G^{i\overline{i}}|\lambda_i| +1)+C(1+\sqrt{K})\sum_{i=1}^{n}G^{i\overline{i}}.
\end{eqnarray}
Combining \eqref{B9} and \eqref{B10} gives \eqref{bar2}.
\end{proof}

\subsection{Estimates of $K^{-\frac{1}{2}}\sum_{p=1}^{n}
\Big[\partial_{y^p}(u-\underline{u})\Big]^2$}

Differentiating the boundary condition $(u-\underline{u})(t, \zeta(t))=0$,
we have
\begin{eqnarray*}
\partial_{y^{p}}(u-\underline{u})=-\partial_{x^n}(u-\underline{u})\partial_{y^{p}}\zeta, \quad \mbox{on}~ \partial M \cap \overline{M_\delta}.
\end{eqnarray*}
Note that $|\partial_{y^{p}}\zeta|\leq C|t|$ in view of $\partial_{y^{p}}\zeta(0)=\partial_{y^{p}}\rho(0)=0$, hence
\begin{eqnarray*}
(\partial_{y^{p}}(u-\underline{u}))^2\leq C|z|^2 \quad \mbox{on} \quad \partial M \cap \overline{M_\delta}.
\end{eqnarray*}
Combining with \eqref{B7}, we can obtain 
\begin{eqnarray}\label{E1}
&&\frac{1}{\sqrt{K}}G^{i\overline{j}}\partial_i\partial_{\overline{j}}
\Big[\partial_{y^p}(u-\underline{u})\Big]^2
\\&=&\frac{2}{\sqrt{K}}G^{i\overline{j}}\partial_i
\partial_{y^p}(u-\underline{u})\partial_{\overline{j}}\partial_{y^p}(u-\underline{u})
+\frac{2}{\sqrt{K}}
\partial_{y^p}(u-\underline{u})
G^{i\overline{j}}\partial_i\partial_{\overline{j}}\partial_{y^p}(u-\underline{u})
\nonumber\\&\geq&\frac{2}{\sqrt{K}}G^{i\overline{j}}\partial_i
\partial_{y^p}(u-\underline{u})\partial_{\overline{j}}\partial_{y^p}(u-\underline{u})
-C|G^{i\overline{j}}\partial_i\partial_{\overline{j}}\partial_{y^p}u|-C\sum_{i=1}^{n}G^{i\overline{i}}
\nonumber\\ \nonumber&\geq&\frac{2}{\sqrt{K}}G^{i\overline{j}}\partial_i
\partial_{y^p}(u-\underline{u})\partial_{\overline{j}}\partial_{y^p}(u-\underline{u})-CG^{i\overline{i}}(1+|\lambda_i|)
-C\sum_{i} G^{i\bar{i}} -C.
\end{eqnarray}

\subsection{Estimates of $\frac{1}{\sqrt{K}}\sum_{a=1}^{n-1}|\nabla_{a}(u-\underline{u})|^2$}

For each $a \in \{1, 2, ..., n-1\}$, we define local sections of $T^{1, 0}M$ around the origin
\begin{eqnarray*}
E_a(0)=\frac{\partial}{\partial z^a}-\Big[\frac{\partial_{z^a}\rho}{\partial_{z^n}\rho}\Big]
\frac{\partial}{\partial z^n}, \quad \mbox{for} \quad 1\leq a\leq n-1.
\end{eqnarray*}
Clearly, those are tangential to $\partial M$. Using the metric $\omega$,
we perform the Gram-Schmidt process to obtain a local orthonormal frame $\{e_a\}_{a=1}^{n-1}$
of $T^{1, 0}M$. Thus, $\{e_a\}_{a=1}^{n-1}$ are tangential to $\partial M$, $\omega(e_a, \overline{e}_b)=\delta_{ab}$ and
\begin{eqnarray*}
e_a(0)=\frac{\partial}{\partial z^a}, \quad \mbox{for} \quad 1\leq a\leq n-1.
\end{eqnarray*}
Furthermore, let
\begin{eqnarray*}
e_n=\frac{E_n}{|E_n|_{\omega}}, \quad E_n=\frac{\partial}{\partial z^n}
-\sum_{a=1}^{n-1}\omega(\partial_n, \overline{e}_a)e_a.
\end{eqnarray*}
Thus, 
\begin{eqnarray}\label{B4-1}
\frac{1}{\sqrt{K}}G^{i\overline{j}}\partial_i \partial_{\overline{j}}|\nabla_a(u-\underline{u})|^2&=&\frac{1}{\sqrt{K}}G^{i\overline{j}}\partial_i \partial_{\overline{j}}
\nabla_a(u-\underline{u})\overline{\nabla_a(u-\underline{u})}
\nonumber\\&=&\frac{1}{\sqrt{K}}G^{i\overline{j}}\partial_i \partial_{\overline{j}}
\Big[e_{a}^{p}\overline{e_{a}^{q}}\partial_p(u-\underline{u}) \partial_{\overline{q}}(u-\underline{u})\Big]
\nonumber\\&\geq&\frac{1}{\sqrt{K}}G^{i\overline{j}}
e_{a}^{p}\partial_{\overline{j}}\partial_{p}(u-\underline{u})\overline{e_{a}^{q}}\partial_{\overline{q}}\partial_i(u-\underline{u})
\nonumber\\&&+\frac{1}{\sqrt{K}}G^{i\overline{j}}
e_{a}^{p}\partial_i\partial_p(u-\underline{u})\overline{e_{a}^{q}}\partial_{\overline{j}}\partial_{\overline{q}}(u-\underline{u})\nonumber\\&&
+\frac{2}{\sqrt{K}}\mathrm{Re} \left( G^{i\overline{j}}
\partial_{\overline{j}}(e_{a}^{p}\overline{e_{a}^{q}})\partial_p\partial_i(u-\underline{u})\partial_{\overline{q}}(u-\underline{u}) \right)\nonumber
\\&&-CG^{i\overline{i}}(1+|\lambda_i|)
-C(1+\sqrt{K})\sum_{i=1}^{n}G^{i\overline{i}}-C.
\end{eqnarray}
We deal with the first term on the right side of the inequality \eqref{B4-1} by
\begin{eqnarray}\label{B4-2}
G^{i\overline{j}}
e_{a}^{p}\partial_{\overline{j}}\partial_p(u-\underline{u})\overline{e_{a}^{q}}\partial_i\partial_{\overline{q}}(u-\underline{u})
&=&G^{i\overline{j}}
(\chi-\chi_{\underline{u}})_{\overline{j}a}(\chi-\chi_{\underline{u}})_{\overline{a}i}
\nonumber\\&\geq&\frac{1}{2}G^{i\overline{j}}
\chi_{\overline{j}a}\chi_{\overline{a}i}-C\sum_{i=1}^{n}G^{i\overline{i}}.
\end{eqnarray}
Next, we can rewrite the third term on the right side of the inequality \eqref{B4-1} as
\begin{eqnarray}\label{B4-3}
\nonumber\frac{2}{\sqrt{K}}G^{i\overline{j}}
\partial_{\overline{j}}(e_{a}^{p}\overline{e_{a}^{q}})\partial_p\partial_i(u-\underline{u})\partial_{\overline{q}}(u-\underline{u})
&=&\frac{2}{\sqrt{K}}G^{i\overline{j}}
\partial_{\overline{j}}(e_{a}^{p}\overline{e_{a}^{q}})\partial_{\overline{p}}\partial_i(u-\underline{u})\partial_{\overline{q}}(u-\underline{u})
\\&&-\sqrt{-1}\frac{2}{\sqrt{K}}G^{i\overline{j}}
\partial_{\overline{j}}(e_{a}^{p}\overline{e_{a}^{q}})\partial_{y^p}\partial_i(u-\underline{u})\partial_{\overline{q}}(u-\underline{u}),
\end{eqnarray}
where we used the relation $\frac{\partial}{\partial z^p}=\frac{\partial}{\partial
\overline{z}^p}-\sqrt{-1}\frac{\partial}{\partial y^p}$. Note that
\begin{eqnarray*}
&&\frac{2}{\sqrt{K}}\sum_{p, q}\Big|G^{i\overline{j}}
\partial_{\overline{j}}(e_{a}^{p}\overline{e_{a}^{q}})\partial_{y^p}\partial_i(u-\underline{u})\partial_{\overline{q}}(u-\underline{u})\Big|
\\&\leq&2\sum_{p, q}\Big[G^{i\overline{j}}
\partial_i\partial_{y^p}(u-\underline{u})\partial_{\overline{j}}\partial_{y^p}(u-\underline{u})
\Big]^{\frac{1}{2}}\Big[G^{i\overline{j}}\partial_{i}(\overline{e_{a}^{p}}e_{a}^{q})
\partial_{\overline{j}}(e_{a}^{p}\overline{e_{a}^{q}})
\Big]^{\frac{1}{2}}\\&\leq&\frac{1}{\sqrt{K}} \sum_i G^{i\overline{j}}
\partial_i\partial_{y^p}(u-\underline{u})\partial_{\overline{j}}\partial_{y^p}(u-\underline{u})
+C\sqrt{K}\sum_{i=1}^{n}G^{i\overline{i}}.
\end{eqnarray*}
Substituting \eqref{B4-2} and \eqref{B4-3} into \eqref{B4-1}, dropping the second
term on the right side of the inequality \eqref{B4-1}, we can obtain
\begin{eqnarray}\label{B4-51}
\nonumber\frac{1}{\sqrt{K}}G^{i\overline{j}}\partial_i \partial_{\overline{j}}|\nabla_a(u-\underline{u})|^2
&\geq&\frac{1}{2\sqrt{K}}G^{i\overline{j}}
\chi_{\overline{j}a}\chi_{\overline{a}i}
-\frac{1}{\sqrt{K}}G^{i\overline{j}}
\partial_i\partial_{y^p}(u-\underline{u})\partial_{\overline{j}}\partial_{y^p}(u-\underline{u})\nonumber
\\&&-C\sum_{i=1}^nG^{i\overline{i}}(1+|\lambda_i|)
-C(1+\sqrt{K})\sum_{i=1}^{n}G^{i\overline{i}}-C.
\end{eqnarray}
By the Lemma 2.7 in \cite{Feng20}, there exists an index $r$ such that
\begin{eqnarray*}
G^{i\overline{j}}
\chi_{\overline{j}a}\chi_{\overline{a}i}\geq\frac{1}{2}\sum_{i\neq r}G^{i\overline{i}}\lambda_{i}^{2}.
\end{eqnarray*}
Going back to \eqref{B4-51}, we get
\begin{eqnarray}\label{B4-5}
&&\frac{1}{\sqrt{K}}G^{i\overline{j}}\partial_i \partial_{\overline{j}}|\nabla_a(u-\underline{u})|^2
\nonumber\\&\geq&\frac{1}{4\sqrt{K}}\sum_{i\neq r}G^{i\overline{i}}\lambda_{i}^{2}
-\frac{1}{\sqrt{K}}\sum_{p=1}^{n}G^{i\overline{j}}
\partial_i\partial_{y^p}(u-\underline{u})\partial_{\overline{j}}\partial_{y^p}(u-\underline{u})\nonumber
\\&&-C\sum_i G^{i\overline{i}}|\lambda_i|-C(1+\sqrt{K})\sum_{i=1}^{n}G^{i\overline{i}}-C.
\end{eqnarray}

\subsection{Mixed tangential-normal estimates}

Combining with \eqref{v}, \eqref{bar2}, \eqref{E1} and \eqref{B4-5}, we obtain
\begin{eqnarray*}
G^{i\overline{j}} \partial_i \partial_{\bar{j}} \Psi
&\leq&-\frac{1}{\sqrt{K}}\sum_{p=1}^{n-1}G^{i\overline{j}}
\partial_i\partial_{y^p}(u-\underline{u})\partial_{\overline{j}}\partial_{y^p}(u-\underline{u})-\frac{1}{4\sqrt{K}}\sum_{i\neq r}G^{i\overline{i}}\lambda_{i}^{2}
\\&&-\frac{\varepsilon }{4}A\sqrt{K} \left(1+\sum_{i=1}^{n}G^{i\overline{i}} \right)+B\sqrt{K}\sum_{i=1}^{n}G^{i\overline{i}}
\\&&+C (1+ \sum_i G^{i\overline{i}}|\lambda_i|)
+C(1+\sqrt{K})\sum_{p=1}^{n}G^{p\overline{p}}
\\&\leq&-\frac{\varepsilon }{4}A\sqrt{K} \left(1+\sum_{i=1}^{n}G^{i\overline{i}} \right)+B\sqrt{K}\sum_{i=1}^{n}G^{i\overline{i}}+C(1+\sqrt{K})\sum_{i=1}^{n}G^{i\overline{i}}+C\\&&
-\frac{1}{4\sqrt{K}}\sum_{i\neq r}G^{i\overline{i}}\lambda_{i}^{2}+CG^{i\overline{i}}|\lambda_i|.
\end{eqnarray*}
Choosing $\frac{\varepsilon}{4}A\geq B+2C+A_0$ with $A_0$ large enough, we have
\begin{eqnarray}\label{La-1}
G^{i\overline{j}}\partial_i \partial_{\bar{j}} \Psi
&\leq&-A_0\sqrt{K}\sum_{i=1}^{n}G^{i\overline{i}}-\frac{1}{4\sqrt{K}}\sum_{i\neq r}G^{i\overline{i}}\lambda_{i}^{2}
+C\sum_iG^{i\overline{i}}|\lambda_i|.
\end{eqnarray}
By \cite[Corollary 2.8]{Guan14},  for any $\varepsilon_0>0$, we have
\begin{eqnarray}\label{La-123}
\sum_i G^{i\bar{i}} |\lambda_i| \leq  \varepsilon_0\sum_{i\neq r}G^{i\overline{i}}\lambda_{i}^{2}+ C_1 (1+\frac{1}{\varepsilon_0}) \sum_i G^{i\bar{i}}.
\end{eqnarray}
Choosing $\varepsilon_0=\frac{1}{4C\sqrt{K}}$  and
substituting the above inequality into \eqref{La-1},  choosing $A_0$ sufficiently large results in
\begin{eqnarray*}\label{La-2}
G^{i\overline{j}}\partial_i \partial_{\bar{j}} \Psi
\leq0.
\end{eqnarray*}

Lastly, we consider the boundary value for $\Psi$ which consists of two pieces.
First, on $\partial M\cap M_{\delta}$, if we take $B$ large enough, we have
\begin{eqnarray}\label{La-3}
\Psi\geq A\sqrt{K}v+B\sqrt{K}|z|^2-C|z|^2\geq 0.
\end{eqnarray}
Secondly, on $\partial M_{\delta}\cap M$, if we take $B$ large enough, we have
\begin{eqnarray}\label{La-4}
\Psi\geq A\sqrt{K}v+B\sqrt{K}\delta^2-C\sqrt{K}\geq 0,
\end{eqnarray}
combining \eqref{La-3} and \eqref{La-4} gives
\begin{eqnarray*}
\Psi\geq 0 \quad \mbox{on} \quad \partial M_{\delta}.
\end{eqnarray*}
Then, applying the maximum principle, we get
\begin{eqnarray*}
\Psi\geq 0 \quad \mbox{in} \quad  M_{\delta}.
\end{eqnarray*}
Note that $\Psi(0)=0$, it yields
\begin{eqnarray*}
\partial_{x^n}\Psi(0)\geq 0.
\end{eqnarray*}
It follows that
$$0\leq A\sqrt{K} \partial_{x^n}v(0)-(\partial_{x^n} \frac{\rho_{t^{\alpha}}}{\rho_{x^n}})(0) \partial_{x^n}(u-\underline{u})(0)
+ \partial_{x^n} \partial_{t^{\alpha}}(u-\underline{u}).$$
Since $|\partial_{x^n}v| \leq C$ on $\partial M$, we conclude
$$\partial_{x^n} \partial_{t^{\alpha}}u(0) \geq -C\sqrt{K}.$$
We can apply the same argument to the function
\begin{eqnarray}\label{mix-fun-02}
\nonumber\widetilde{\Psi}:&=&A\sqrt{K}v+B\sqrt{K}|z|^2-\frac{1}{\sqrt{K}}\sum_{i=1}^{n}
\Big[\partial_{y^i}(u-\underline{u})\Big]^2
-\frac{1}{\sqrt{K}}\sum_{i=1}^{n-1}|\nabla_i(u-\underline{u})|^2
\\&&-T_{\alpha}(u-\underline{u}).
\end{eqnarray}
It follows that
$$\partial_{x^n} \partial_{t^{\alpha}}u(0) \leq C\sqrt{K}.$$
Thus,
\begin{eqnarray*}
|\chi_{\alpha' \overline{n}}|(0)\leq C \sqrt{K}   \quad \forall \alpha'\in\{1, 2, \cdots, n-1\}.
\end{eqnarray*}

\section{Boundary double normal estimate}

For $p \in \partial M$, we choose coordinates $z=(z^1, ..., z^n)$ such that $z(p)=0$ and
$g_{i\overline{j}}=\delta_{ij}$. We denote $z^i=x^i+\sqrt{-1}y^i$ and rotate the coordinates such that
$\frac{\partial}{\partial x^n}$ is the unit inner normal vector at $p$. Then, we perform an orthogonal change
of coordinates in the tangential directions to arrange that
\begin{eqnarray*}
\chi_{i\overline{j}}=(\chi_0)_{i\overline{j}}+u_{i\overline{j}}=\lambda^{\prime}_{i}\delta_{ij} \quad
\mbox{for} \quad 1\leq i, j\leq n-1.
\end{eqnarray*}
Since $\chi \in \Gamma_{k-1}(M)$, we have
\begin{eqnarray*}
\chi_{n\overline{n}}+\sum_{i=1}^{n-1}\lambda^{\prime}_{i}\geq 0,
\end{eqnarray*}
which implies
\begin{eqnarray*}
\chi_{n\overline{n}}\geq -C.
\end{eqnarray*}
Thus, it remains to estimate $\chi_{n\overline{n}}$ from above.  We give the following lemma before starting the estimate.

\begin{lemma}
Let $A = (a_{i\bar{j}})$ be a $n\times n$  Hermitian matrix and $A'=(a_{i\bar{j}})_{1\leq i, j\leq n-1}$ be the $(n-1)\times (n-1)$ Hermitian matrix. Suppose that $\lambda_1(A) \leq \cdots \leq \lambda_n(A)$ are  the eigenvalues of $A $  and $\lambda'_1(A') \leq \cdots \leq \lambda_{n-1}'(A')$ are the
eigenvalues of  $A'$. Then we have
\begin{eqnarray}\label{mix-nn-lem-1}
\lambda_j(A)\leq \lambda_j'(A') \leq \lambda_{j+1} (A), \quad 1\leq j\leq n-1,
\end{eqnarray}
and
\begin{eqnarray}\label{mix-nn-lem-2}
\left\{
\begin{aligned}
&\lambda_j(A)=\lambda_j'(A')+o(1) , && 1\leq j \leq n-1,\\
&a_{n\bar{n}}\leq \lambda_n(A)\leq a_{n\bar{n}}\left(1+ O(\frac{1}{a_{n\bar{n}}}) \right).
\end{aligned}
\right.
\end{eqnarray}
 as $|a_{n\bar{n}}| \rightarrow +\infty$.
\end{lemma}

\begin{proof}
\eqref{mix-nn-lem-1} can be  obtained by the  Cauchy’s interlace inequality (see for example \cite{Hwang04}) and \eqref{mix-nn-lem-2} is follows from \cite[Lemma 1.2]{CNS85}
\end{proof}

We set
\begin{eqnarray*}
\Gamma_{\infty}:=\{(\lambda_1, \cdots, \lambda_{n-1}) \mid ((\lambda_1, \cdots, \lambda_{n-1}, \lambda_n)\in \Gamma_{k-1} ~\mbox{for some} ~\lambda_n )\},
\end{eqnarray*}
and
\begin{eqnarray*}
f(\lambda(z)):=\frac{\sigma_k(\lambda(z))}{\sigma_{k-1}(\lambda(z))}- \sum_{l=0}^{k-1} \beta_l(z) \frac{\sigma_l(\lambda(z))}{\sigma_{k-1}(\lambda(z))}
\end{eqnarray*}
for any continuous function $\lambda$. For any $(n-1)\times (n-1)$ Hermitian matrix $E$ with $\lambda'(E)\in \Gamma_{\infty}$, we define
$$\widetilde{G}(E)= f_{\infty}(\lambda'(E)):=\lim_{\lambda_n \rightarrow \infty} f(\lambda_1'(E), \cdots, \lambda_{n-1}'(E), \lambda_n).$$

\begin{proof}[\textbf{The upper bound estimate of $\chi_{n\overline{n}}$}]

For simplicity, we denote by  $\lambda=\lambda(\chi_u(z))$ and $\lambda'=\lambda((\chi_u'(z))$, where $\chi_u'(z)=((\chi_u)_{i\bar{j}})_{1\leq i, j \leq n-1}$.
The proof is devided into two claims:


\textbf{Claim 1:}
\begin{eqnarray}\label{mix-nn-ineq-32}
P_{\infty}:=\min_{z\in \partial M} \left(\widetilde{G}(\chi_u'(z))-\beta(z) \right) > c_0
\end{eqnarray}
for some uniform constant $c_0$.

We assume that $P_{\infty}=\widetilde{G}(\chi_u'(z_0))-\beta(z_0)$ at point   $z_0\in \partial M$.
Since $\lambda(\chi_u)\in \Gamma_{k-1}$ and $\sigma_k(\lambda)=\lambda_n \sigma_{k-1} (\lambda \mid i)+\sigma_k(\lambda \mid i)$, then
$$\widetilde{G}(\chi_u'(z))=\frac{\sigma_{k-1}}{\sigma_{k-2}}(\lambda_1', \cdots, \lambda_{n-1}')-\sum_{l=1}^{k-2} \beta_l(z) \frac{\sigma_{l-1}}{\sigma_{k-2}}(\lambda_1', \cdots, \lambda_{n-1}').$$
Denote $\widetilde{G}^{i\bar{j}}_0:=\frac{\partial \widetilde{G}}{\partial (\chi'_u)_{i\bar{j}}}(\chi'_u(z_0))$ for $1\leq i, j\leq n-1$. By the concavity of $G$ and $\widetilde{G}$, we have
\begin{eqnarray}\label{mix-nn-ineq-133}
\nonumber&&\widetilde{G}^{i\bar{j}}_0((\chi'_u)_{i\bar{j}}(z)-(\chi'_u)_{i\bar{j}}(z_0))\\
\nonumber&=& \left(\frac{\sigma_{k-1}}{\sigma_{k-2}}\right)^{i\bar{j}}\vert_{z_0}((\chi'_u)_{i\bar{j}}(z)-(\chi'_u)_{i\bar{j}}(z_0))+\sum_{l=1}^{k-2} \beta_l(z_0) \left(-\frac{\sigma_{l-1}}{\sigma_{k-2}}\right)^{i\bar{j}} ((\chi'_u)_{i\bar{j}}(z)-(\chi'_u)_{i\bar{j}}(z_0))\\
\nonumber&\geq&  \frac{\sigma_{k-1}}{\sigma_{k-2}}(\chi'_u(z))-\frac{\sigma_{k-1}}{\sigma_{k-2}}(\chi'_u(z_0)) + \sum_{l=1}^{k-2} \beta_l(z_0) \left(-\frac{\sigma_{l-1}}{\sigma_{k-2}}(\chi'_u(z))+ \frac{\sigma_{l-1}}{\sigma_{k-2}}(\chi'_u(z_0))  \right)\\
\nonumber&\geq&  \widetilde{G}(\chi_u'(z))-\widetilde{G}(\chi_u'(z_0))-C \sum_{l=1}^{k-2} |\beta_l|_{C^1} |z-z_0|,
\end{eqnarray}
and
\begin{eqnarray}\label{mix-nn-ineq-134}
\nonumber&&\widetilde{G}^{i\bar{j}}_0(\chi'_u)_{i\bar{j}}(z)-\beta(z) - \widetilde{G}^{i\bar{j}}_0(\chi'_u)_{i\bar{j}}(z_0)+\beta(z_0)\\
\nonumber&\geq&  \widetilde{G}(\chi_u'(z))- \beta(z)-c_0-C \sum_{l=1}^{k-2} |\beta_l|_{C^1} |z-z_0|\\
&\geq& -C \sum_{l=1}^{k-2} |\beta_l|_{C^1} |z-z_0|.
\end{eqnarray}
Note that
\begin{eqnarray}\label{mix-nn-ineq-135}
c_{\infty}:= \min_{z\in \partial M} \left(\widetilde{G}(\chi_{\underline{u}}'(z))-G(\chi_{\underline{u}}(z)) \right) >0.
\end{eqnarray}
Using \eqref{Bd-DD}, we have
\begin{eqnarray*}
(\chi_u)_{i\bar{j}}(z_0)
=(\chi_{\underline{u}})_{i\bar{j}}(z_0)
-\partial_{x^n}(u-\underline{u})(z_0) \zeta_{i\bar{j}}(z_0)
\end{eqnarray*}
for any $1\leq i, j \leq n-1$. Then
\begin{eqnarray}\label{mix-nn-ineq-136}
\partial_{x^n}(u-\underline{u})(z_0) \sum_{i,j=1}^{n-1} \zeta_{i\bar{j}}(z_0) \widetilde{G}^{i\bar{j}}_0 (z_0) &=& \widetilde{G}^{i\bar{j}}_0 \left( (\chi'_{\underline{u}})_{i\bar{j}}(z_0)-(\chi_u')_{i\bar{j}} (z_0)\right)\\
\nonumber&\geq& \widetilde{G} (\chi_{\underline{u}}(z_0))- \widetilde{G} (\chi_{u}(z_0))\\
\nonumber&=& \widetilde{G} (\chi_{\underline{u}}(z_0))-\beta(z_0) -c_0\\
\nonumber&\geq& \widetilde{G} (\chi_{\underline{u}}(z_0))-G(\chi_{\underline{u}}(z_0)) -c_0\\
\nonumber&\geq& c_{\infty}-c_0.
\end{eqnarray}
Consequently, if $\partial_{x^n}(u-\underline{u})(z_0)\sum_{i,j=1}^{n-1} \zeta_{i\bar{j}}(z_0) \widetilde{G}^{i\bar{j}}_0 (z_0)  \leq \frac{c_{\infty}}{2}$, then $c_0 \geq \frac{c_{\infty}}{2}$ and we are done.

Suppose now that
$$\partial_{x^n}(u-\underline{u})(z_0)\sum_{i,j=1}^{n-1} \zeta_{i\bar{j}}(z_0) \widetilde{G}^{i\bar{j}}_0 (z_0) \geq \frac{c_{\infty}}{2}.$$
Let $\eta(z)= \sum_{i,j=1}^{n-1} \zeta_{i\bar{j}}(z) \widetilde{G}^{i\bar{j}}_0 (z)$. Note that
$$\eta(z_0) \geq \frac{c_{\infty}}{2 \partial_{x^n}(u-\underline{u})(z_0)} \geq 2 \epsilon_1 c_{\infty}$$
for some uniform $\epsilon_1>0$ since  \eqref{mix-c1-b-est}. We may assume that
$$\eta \geq \epsilon_1 c_{\infty} \quad on~ \bar{M}_{\delta}(z_0)$$
by requiring $\delta$ small, where $M_{\delta}(z_0):=\{z\in M \mid \mbox{dist}_{g_0} (z, z_0) < \delta\}$.  We consider the function
\begin{eqnarray*}
\Phi(z) &=&-\partial_{x^n}(u-\underline{u}) (z)+\frac{1}{\eta(z)}  \sum_{i,j=1}^{n-1}  \widetilde{G}^{i\bar{j}}_0 \left( (\chi'_{\underline{u}})_{i\bar{j}}(z)-(\chi_u')_{i\bar{j}}(z_0) \right) - \frac{\beta(z)-\beta(z_0)}{\eta(z)} \\
&&+\frac{C}{\eta(z)}  \sum_{l=1}^{k-2} |\beta_l|_{C^1} |z-z_0|.
\end{eqnarray*}
We deduce from \eqref{mix-nn-ineq-134} and the fact $(\chi_u)_{i\bar{j}}(z)
=(\chi_{\underline{u}})_{i\bar{j}}(z)
-\partial_{x^n}(u-\underline{u})(z) \zeta_{i\bar{j}}(z)$ on $\partial M \cap \bar{M}_{\delta}(z_0)$ that
\begin{eqnarray*}
\Phi(z_0)=0, \quad   \Phi(z) \geq 0 \quad \forall z\in \partial M \cap \bar{M}_{\delta}(z_0),
\end{eqnarray*}
while by \eqref{bar2} and \eqref{B8},
\begin{eqnarray*}
G^{i\bar{j}}\Phi_{i\bar{j}}\leq C \sqrt{K} (1+ \sum_{i} G^{i\bar{i}} )+ C \sum_i G^{i\bar{i}} |\lambda_i(\chi_u)|^2).
\end{eqnarray*}
Consider the function $\widetilde{\Psi}$ defined in \eqref{mix-fun-02}, it follows from \eqref{v}, \eqref{bar2}, \eqref{E1}, \eqref{B4-5} and \eqref{La-1} that
\begin{eqnarray*}
\left\{
\begin{aligned}
& G^{i\overline{j}}(\Phi+\widetilde{\Psi})_{i\bar{j}} \leq 0, && in~ M_{\delta}(z_0),\\
&\Phi+\widetilde{\Psi}\geq 0 && on~ \partial M_{\delta}(z_0).
\end{aligned}
\right.
\end{eqnarray*}
The maximum principle yields that $\Phi+\widetilde{\Psi}\geq 0 $ in $\partial M_{\delta}(z_0)$, and then $\partial_{x_n} \Phi(z_0)=-\phi_{\nu}(z_0) \geq \widetilde{\Psi}_{\nu}(z_0)$. This, together with the difinition of $\Phi$, yields that
$$\partial_{x^n}\partial_{x^n} u(z_0) \leq C K.$$
Since
$$u_{\bar{n}n}= \frac{1}{4} (\frac{\partial}{\partial x^n} + \sqrt{-1} \frac{\partial}{\partial y^n})  (\frac{\partial}{\partial x^n} - \sqrt{-1} \frac{\partial}{\partial y^n}) u, $$
it follows that  $\lambda(\chi_u)(z_0)$ is contained in a compact subset of $\Gamma$. Therefore
$$P_{\infty}\geq f(\lambda'(\chi_u')(z_0), R)-\beta(z_0)>0$$
for $R$ sufficiently large since $f_i >0$, which yields \eqref{mix-nn-ineq-32}.

\textbf{Claim 2:}  There holds
\begin{eqnarray}\label{mix-nn-ineq-33}
\chi_{n\overline{n}}(z) \leq CK    \quad z\in \partial M
\end{eqnarray}
for some constant $C$.

Let $z\in \partial M$, we  suppose that $\lambda_1 \leq \cdots \leq \lambda_n$ and $\lambda_1' \leq \cdots \leq \lambda_{n-1}'$.
From the boundary tangential-tangential and tangential-normal estimates, it follows that $\lambda'(\chi'_u)$ lies in a compact set $L\subset \Gamma_{\infty}$. Combining with \eqref{mix-nn-ineq-32}, we know that  there exist uniform positive constant $c_0$ and $R_0$ deponding on the range of $\lambda'(\chi'_u)$
 such that  for any $R>R_0$
$$f(\lambda', R) >\sup_M \beta +c_0,$$
which implies
\begin{eqnarray}\label{mix-nn-ineq-21}
f(\widetilde{\lambda}, R) > \sup_M \beta + \frac{c_0}{2}
\end{eqnarray}
for any $\widetilde{\lambda}\in U_L$ and $R>R_0$, where $U_L$ is the neighborhood of $L$.

Assuming that there exists a constant $R_1>R_0$ such that $\chi_{n\bar{n}}(z) \geq R_1$.
According to \eqref{mix-nn-lem-1} and \eqref{mix-nn-lem-2}, it is esay to see that
$$\lambda_n\geq \chi_{n\bar{n}}(z) \geq R_0, \quad \lambda_j \in U_L, \quad 1\leq j\leq n-1.$$
Combining with \eqref{mix-nn-ineq-21}, we know that
$G(\chi_u(z))= f(\lambda)> \beta(z) +\frac{c_0}{2},$
This contradicts with the equation \eqref{K-eq1}, and hence \eqref{mix-nn-ineq-33} holds.
\end{proof}

Combining this with the second order interior estimate, i.e., Theorem \ref{inter-c2}, we have
\begin{theorem}\label{mix-C2}
Let $u\in C^{\infty}(\bar{M})$ be an $(k-1)$-admissible solution for equation \eqref{K-eq}.
Under the assumptions mentioned  in Theorem \ref{Main},  then there exists a positive constant $C$ depending only on  $(M, \omega), \chi_0$, $\alpha_l, \varphi$ and the subsolution $\underline{u}$ such that
 $$\sup_{\bar{M}} |\sqrt{-1} \partial \bar{\partial} u|\leq C K,$$
 where $K:=1+ \sup_M  |\nabla u|^2.$
 \end{theorem}

\section{ Gradient estimates}

In this section, we combine the second derivative estimate with a blow-up argument and Liouville type theorem due to Dinew-Kolodziej \cite{Din12} to obtain gradient estimates.

\begin{theorem}\label{C1}
Let $u\in C^{\infty}(\bar{M})$ be an $(k-1)$-admissible solution for equation \eqref{K-eq}.
Under the assumptions mentioned  in Theorem \ref{Main},  then there exists a positive constant $C$ depending only on  $(M, \omega), \chi_0$, $\alpha_l, \varphi$ and the subsolution $\underline{u}$ such that
 $$\sup_{M} |\nabla u|\leq C.$$
 \end{theorem}

\begin{proof}
 Suppose that there exists a sequence of function $u_m\in C^4(\bar{M})$ to the equation \eqref{K-eq} such that
 $$N_m=\sup_{M} |\nabla u_m| \rightarrow +\infty \quad \mbox{as} ~m \rightarrow +\infty.$$
 Hence
\begin{eqnarray}\label{mix-gra-ineq-09}
\chi^{k}_{u_m}\wedge \omega^{n-k}=\sum_{l=0}^{k-1}\alpha_l(z)\chi^{l}_{u_m}\wedge \omega^{n-l}, \quad u=\varphi ~\mbox{on}~ \partial M.
\end{eqnarray}
For any $m$, we assume that $|\nabla u_m|$ attains its maximum vulue at $z_m\in M$. Then, after passing a subsequence we may assume that $z_m \rightarrow z'$ for some point $z'\in \bar{M}$. By Theorem \ref{mix-C2} we have
\begin{eqnarray}
\sup_{M} |\sqrt{-1} \partial \bar{\partial} u_m|\leq C (1+ N_m^2),
\end{eqnarray}
where the constant independs of $m$. we divide our proof into two cases separately.

\textbf{Case 1}: $z'\in \overset{\circ}{M}$.

Choosing a small coordinate ball centered at $z'$, which we identify with an open set in $\mathbb{C}^n$ with coordinates $(z^1, \cdots, z^n)$, and such that 
$\omega(0)=\omega_0=\frac{\sqrt{-1}}{2} \delta_{ij} dz^i \wedge d \bar{z}^{j}$.  We can assume that all $z_m$ are within this coordinate ball. Let $R>0$, we define
$$v_m(z):= u_m( \frac{z}{N_m}+z_m), \quad \forall z \in B_R(0),$$
which is well-defined for any $m$ such that $N_m$ is large enough. Clearly,
$$ |\nabla v_m(0)|=1,  \quad |v_m|_{C^2(B_R(0)) }   \leq C.$$
Let $R \rightarrow +\infty$ and taking a diagonal subsequence again, we can assume that $v_m \rightarrow v$ in $C^{1, \frac{\gamma}{2}}_{\mbox{loc}} (\mathbb{C}^n)$ with $\nabla v(0)=1$.
Then we have from \eqref{mix-gra-ineq-09}
\begin{eqnarray*}
&&\left[\chi_0(\frac{z}{N_m}+z_m)+N_m^2 \frac{\sqrt{-1}}{2} \partial \bar{\partial} v_m\right]^k\wedge \left[\omega(\frac{z}{N_m}+z_m)\right]^{n-k}\\
&&=\sum_{l=0}^{k-1}\alpha_l(\frac{z}{N_m}+z_m)\left[\chi_0(\frac{z}{N_m}+z_m)+N_m^2 \frac{\sqrt{-1}}{2} \partial \bar{\partial} v_m\right]^l\wedge\left[\omega(\frac{z}{N_m}+z_m)\right]^{n-l}.
\end{eqnarray*}
We can take a limit of the equation and obtian
$$\left(\sqrt{-1} \partial \bar{\partial} v \right)^k \wedge \omega_0^{n-k}=0,$$
which is in the pluripotential sense. Moreover, we have for any $1\leq l \leq k-1$ by a similar reasoning
$$\left(\sqrt{-1} \partial \bar{\partial} v \right)^l \wedge \omega_0^{n-l}\geq 0.$$
Combining with the result of Blocki \cite{Blo05}, we know that $v$ is a maximal $k$-subharmonic function in $\mathbb{C}^n$. Then the Liouville thereom in \cite{Din12} implies that $v$ is a constant, which contradicts the fact $\nabla v(0)=1$.

\textbf{Case 2}: $z'\in \partial M$.

Let $\Omega \subset M$ be a coordinate chart centered at $z'$.
Then there exists a smooth function $\rho : B_{2s} \rightarrow \mathbb{R}$ such that
$$\partial M \cap \Omega =\{\rho=0\}, \quad M  \cap \Omega\subset \{\rho \leq 0\}.$$
Where $B_{2s} \subset \mathbb{C}^{2n}$ is the  ball of radius $2s$ at $0=z'$.
Without loss of generality, we may assume $z_m\in \Omega$ with $|z_m|<s$ in local coordinates and
$$r_m:=\mbox{dist}(z_m, \partial M \cap \Omega)=|z_m-y_m|$$
for a unique point $y_m\in \partial M \cap \Omega$. Clearly, $y_m \rightarrow z$ as $m\rightarrow +\infty$. Set
$$v_m(z):= u_m( \frac{z}{N_m}+z_m), \quad  z\in \Omega_m,$$
where $\Omega_m:=\{z\in\Omega \mid \frac{z}{M_m}+z_m\in B_{2s}\cap \{\rho\leq 0\} \}$. Therefore
$$ |\nabla v_m(0)|=1,  \quad |u_m|_{C^2(\Omega_m) }   \leq C.$$
Let $\rho_m:=\rho(\frac{z}{N_m}+z_m)$, then $B_{sN_m} \cap \{ \rho_m \leq 0\} \subset \Omega_i$ since $|z_m|<s$. By the standard elliptic theory, we know that
\begin{eqnarray}
|v_m|_{C^{2,\gamma}\left(B_{\frac{sN_m}{2}} \cap \{ \rho_m \leq 0\} \right)} \leq C.
\end{eqnarray}
Then we divide the proof into the following two sub-cases

\textbf{Case 2.1}
$$\liminf_{m\rightarrow \infty} N_m r_m =+\infty.$$
After passing to a subsequence $v_m$ converges in $C^{2,\gamma}$ on compact sets to $v\in C^{2,\gamma}(\mathbb{C}^n)$ since $N_mr_m \rightarrow \infty$ as $m\rightarrow \infty$. Similar to  Case 1,  we know that $v$ is  a constant, which  contradicts the fact $\nabla v(0)=1$.

\textbf{Case 2.2}
$$\liminf_{m\rightarrow \infty} N_m r_m =L\in[0, \infty).$$

The proof for Case 2.2 is quite similar to that given for Case 2b in \cite[Proposition 6.1 ]{Co19},  and so they are omitted.
\end{proof}


Last we apply the standard continuity method to solve the Dirichlet problem  \ref{K-eq}. 
\begin{proof}[\textbf{Proof of Theorem 1.1}]
For any $t\in [0, 1]$, we consider the equation
\begin{equation}\label{K-eq-12}
\left\{
\begin{aligned}
& G(\chi_{u_t})=t \beta +(1-t) G(\chi_{\underline{u}}), && in~ \bar{M},\\
&u_t=\varphi && on~ \partial M,
\end{aligned}
\right.
\end{equation}
where
$$G(\chi_{u_t})=\frac{ \sigma_k(\chi_{u_t})}{\sigma_{k-1}(\chi_{u_t})}-\sum_{l=0}^{k-2} \beta_l(z) \frac{\sigma_l(\chi_{u_t})}{\sigma_{k-1}(\chi_{u_t})}.
$$
Set
\begin{eqnarray*}
\mathbf{S}=\{t\in [0,1]\mid \mbox{there exists }~u_t\in C^{4, \alpha}(\bar{M})~\mbox{with}~\lambda(\chi_{u_t})\in \Gamma_{k-1}~solving  \eqref{K-eq-12} \}.
\end{eqnarray*}
Clearly the set $\mathbf{S}$ is non-empty since $\underline{u}$ solve the equation \eqref{K-eq-12} for $t=0$. On the one hand, by the generalized Newton-MacLaurin inequality, we have
$$\sigma_{k-1}(\chi_{u_t})\geq C,$$
which implies that the equation \eqref{K-eq-12} is uniformly elliptic. Therefore $\mathbf{S}$ is an open set by the implicit function theorem.

On the other hand, let $u_{t_i}$ be the solution of \eqref{K-eq-12} for $t_i \in [0, 1]$ with $t_i \rightarrow t_0$. From \eqref{mix-c1-b-est},  theorem \ref{C0} , \ref{mix-C2} and \ref{C1}, we know that
$$|u_{t_i}|_{C^2(M)} \leq C$$
for some uniformly constants. Then, higher-order estimates follow from the Evan-Krylov theorem and the Schauder estimate, i.e.,
$$|u_{t_i}|_{C^{4,\alpha}(M)} \leq C.$$
We can take convergent subsequence to a limiting function $\widetilde{u}\in C^{4, \alpha}$ which solves the equation \eqref{K-eq-12} for $t_0$. Hence $\mathbf{S}$ is also closed and we complete the proof of Theorem \ref{Main}.
\end{proof}

 \bibliographystyle{siam}

\end{document}